\documentclass[oneside,10pt]{amsart}	

\mathchardef\emptyset="001F

\newcommand{\e}{\varepsilon}
\newcommand{\meno}{\setminus}

\usepackage{geometry}
\usepackage{mathrsfs}
\usepackage{enumerate}
\usepackage{mathtools}
\usepackage{color}
\usepackage{amsthm}
\usepackage[colorlinks=true,urlcolor=blue,
citecolor=red,linkcolor=blue,
linktocpage,pdfpagelabels,bookmarksnumbered,bookmarksopen
]{hyperref}

\usepackage{pinlabel}
\usepackage{nextpage}

 \allowdisplaybreaks

\newtheorem{teo}{Theorem}[section]
\newtheorem{cor}[teo]{Corollary}
\newtheorem{prop}[teo]{Proposition}
\newtheorem{lm}[teo]{Lemma}
\theoremstyle{definition}
\newtheorem{dfn}[teo]{Definition}
\theoremstyle{remark}
\newtheorem{rmq}[teo]{Remark}
\newtheorem{example}[teo]{Example}

\newtheorem*{ack}{Acknowledgments}

\numberwithin{equation}{section}

\def\XXint#1#2#3{{\setbox0=\hbox{$#1{#2#3}{\int}$} \vcenter{\vspace{-1pt}\hbox{$#2#3$}}\kern-.5\wd0}}
\def\Xint#1{\mathchoice {\XXint\displaystyle\textstyle{#1}}{\XXint\textstyle\scriptstyle{#1}}{\XXint\scriptstyle\scriptscriptstyle{#1}}{\XXint\scriptscriptstyle\scriptscriptstyle{#1}}\!\int}
\def\intmed{\Xint{-}}

\def\Limsup{\mathop{\smash\limsup\vphantom\liminf}}

\newcommand{\mres}{\mathbin{\vrule height 1.6ex depth 0pt width
0.13ex\vrule height 0.13ex depth 0pt width 1.3ex}}

\title[Transmission conditions obtained by homogenisation]{Transmission conditions obtained by homogenisation}

\author{Gianni Dal Maso}
\address[G.\@ Dal Maso]{Scuola Internazionale Superiore di Studi Avanzanti, Trieste, Italy}
\email{dalmaso@sissa.it}

\author{Giovanni Franzina}
\address[G.\@ Franzina]{Istituto Nazionale di Alta Matematica, Unit\`a di Ricerca di Firenze - DiMaI ``U. Dini'', Universit\`a di Firenze, Italy}
\email{giovanni.franzina@unifi.it}

\author{Davide Zucco}
\address[D.\@ Zucco]{Dipartimento di Matematica ``G. Peano'', Universit\`a di Torino, Italy}
\email{davide.zucco@unito.it}
\thanks{Preprint SISSA 11/2018/MATE}

\begin{document}

\maketitle

\begin{abstract}
Given a bounded open set in $\mathbb{R}^n$, $n\ge 2$, and a sequence $(K_j)$ of compact sets converging to an $(n-1)$-dimensional manifold $M$, we study the asymptotic behaviour of the solutions to some minimum problems for integral functionals on $\Omega\setminus K_j$, with Neumann boundary conditions on $\partial(\Omega\setminus K_j)$. We prove that the limit of these solutions is a minimiser of the same functional on $\Omega\setminus M$ subjected to a transmission condition on $M$, which can be expressed through a measure $\mu$ supported on $M$. The class of all measures that can be obtained in this way is characterised, and the link between the measure $\mu$ and the sequence $(K_j)$ is expressed by means of suitable local minimum problems.
\end{abstract}

\bigskip

\bigskip

\noindent{\textbf{Keywords.} $\Gamma$-convergence, capacitary measures, Neumann sieve.}


\section{Introduction}

This paper studies the asymptotic behaviour of the solutions $u_j$ to the equations
\[
	-\Delta u_j +u_j=h \qquad \text{in $\Omega\setminus K_j$,}
\]
with homogeneous Neumann boundary conditions on $\partial (\Omega\setminus K_j)$.
Here and henceforth $\Omega$ is a bounded open set in  $\mathbb{R}^n$, $n\ge 2$,
and $(K_j)$ is a sequence of compact sets in $\mathbb{R}^n$. We assume that there exists a
compact $(n-1)$-dimensional $C^1$ manifold $M$
with boundary, contained in $\Omega$,
such that
\begin{equation}
\label{eq:1.1}
K_j \subset \big\{ x\in\mathbb{R}^n\,\,\colon\,\, {\rm dist}(x,M)\le\rho_j\big\}\,,
\end{equation}
for a suitable sequence of positive numbers $(\rho_j)$ with $\rho_j\to0^+$ as $j\to\infty$. 

It is well known that $u_j$ is the minimiser of the functional
\begin{equation*}
	\int_{\Omega\setminus K_j}|\nabla u|^2dx +\int_{\Omega\setminus K_j} u^2dx - 2\int_{\Omega\setminus K_j}hu\, dx\,,
\end{equation*}
and this property will be the starting point of our analysis.

More in general, in this paper we will consider sequences of minimisers of functionals
of the form
\begin{equation}
\label{fct0}
\int_{\Omega\meno K_j} f(\nabla u)\,dx  +\int_{\Omega\setminus K_j} g(x,u)\,dx\, ,
\end{equation}
where $f$ and $g$ are suitable functions satisfying standard convexity and growth conditions
(see Section~\ref{subsec:2.1} for details).  A significant instance included in our analysis will be that of the functions $f(\nabla u)=|\nabla u|^p$, with $1<p\le n$, and $g(x,u)=|u-h(x)|^q$, with $1\le q<\infty$ and $h\in L^q(\Omega)$.
\medskip

This kind of question is related with the so-called \emph{Neumann sieve problem} that was proposed by Sanchez and Palencia, who gave in \cite{SP2} a formal asymptotic expansion of the solution. It was then studied by Attouch, Damlamian, Murat, and Picard (see \cite{AP, D, M, P})  in the case where the perforations are composed of open balls periodically distribuited over the manifold. For related studies on the asymptotic behaviour of periodically-perforated domains see \cite{AB, AB2, C1, C2, C3, DV, SP1, SP3}.

\medskip
A slight modification (see Section~\ref{sec:3}) of the results of~\cite{Cor} shows that, under our general conditions, there exists a subsequence of $(K_j)$ (not relabelled) satisfying \eqref{eq:1.1} such that the minimisers $u_j$ of \eqref{fct0} converge in $L^q(\Omega)$ as $j\to\infty$ to the minimiser of a functional of the form
\begin{equation}
\label{eq:1.5}
				\displaystyle 
				\int_{\Omega\meno M} f(\nabla u)\,dx + \int_{M} [u]^p\,{d}\mu  
				+\int_{\Omega\setminus M} g(x,u)\,dx\,, \quad u\in L^{1,p}(\Omega\meno M)\,,
\end{equation}
where $\mu$ is a suitable  Borel measure on $\Omega$, concentrated on $M$ and vanishing on all Borel sets $B$ with $C_{1,p}(B)=0$ (see Section~\ref{subsec:2.2} for the definition of the $p$-capacity $C_{1,p}$ and Section~\ref{subsec:2.3} for details on these $p$-capacitary measures).
Here and henceforth $L^{1,p}$ denotes the Deny-Lions space, while the jump of $u$ is defined  by
\( [u]:=|u^+-u^-|\), where $u^+$ and $u^-$ are the measure-theoretic limits
of $u$ at $x$ on both sides of $M$ (see again Section~\ref{subsec:2.2}). The measure $\mu$ appearing in \eqref{eq:1.5} is independent of $g$ and depends only on the sequence of compact sets $(K_j)$ and on the energy density $f$. 

In some special cases, a suitable choice of the sequence of compact sets $(K_j)$ allows for an explicit computation of the measure $\mu$.
For example, if $K_j=\emptyset$ for all $j$, then one obtains in the limit the measure $\mu$ defined for every Borel set $B$ as 
\[
	\mu(B) = \begin{cases}
	0\,, & \ \text{if } C_{1,p}(B\cap M)=0\,,\\
	\infty\,, &\ \text{otherwise.}
	\end{cases}
\]
In this case, the finiteness of the functional \eqref{eq:1.5} implies that $[u]=0$ on $M$ so that  \eqref{eq:1.5} reduces to
\begin{equation*}
				\displaystyle 
				\int_{\Omega} f(\nabla u)\,dx  +\int_{\Omega}g(x,u)\,dx\,, 
				\quad u\in L^{1,p}(\Omega)\,.
\end{equation*}

Moreover, if $K_j=\{x\in\mathbb R^n\colon {\rm dist}(x,M)\le\rho_j\}$, then $\mu=0$ and the corresponding limit functional \eqref{eq:1.5} becomes
\begin{equation*}
				\displaystyle 
				\int_{\Omega\setminus M} f(\nabla u)\,dx  +\int_{\Omega\setminus M} g(x,u)\,dx\,,
				\quad u\in L^{1,p}(\Omega\setminus M)\,.
\end{equation*}

From our point of view the most interesting case is when the limit measure is of the form
$\mu = \theta\mathscr H^{n-1}\mres M$ for a function $\theta\in L^{p'}(M,\mathscr H^{n-1})$ where $p'=p/(p-1)$  (see Section~\ref{sec:density}).
In this case, under suitable regularity assumptions on $f$
the first order minimality conditions lead to a partial differential equation
in $\Omega\meno M$ with suitable \emph{transmission conditions} across $M$. For instance,
when $f(\nabla u)=|\nabla u|^2$ and $g(x,u)=|u-h(x)|^2$ for some $h\in L^2(\Omega)$,
the minimiser of \eqref{eq:1.5} satisfies
\[
\begin{cases}
	-\Delta u +u=h\,, & \ \text{in $\Omega\setminus M$,}\\
	\partial_\nu u=0\,, & \ \text{on $\partial \Omega$,}\\
	(\partial_\nu u)^+=	(\partial_\nu u)^-=\theta(u^+-u^-) \,,& \ \text{on $M$,} 
\end{cases}
\]
where $\nu$ is a unit normal. A particular case when $\theta$ is constant and $M$ is a hyperplane has been investigated in \cite{ansini}.

In this paper we prove a density result (see Section~\ref{sec:density}), which shows that every measure 
vanishing on sets of $p$-capacity zero can appear in the limit problem \eqref{eq:1.5}. More
precisely, given $f$, $M$, and $\mu$,
we prove that there exists a sequence $(K_j)$ of compact sets, satisfying \eqref{eq:1.1}, such that,
for every $g$, the minimisers of \eqref{fct0} converge in $L^q(\Omega)$
to the minimiser of \eqref{eq:1.5}. For this result the hypothesis $p\le n$ is crucial. 

We also prove an asymptotic formula which allows us to obtain the measure
$\mu$ starting from some auxiliary minimum problems which involve $f$ and the $K_j$'s
(see Section~\ref{reconstruction}). Moreover, for a given sequence $(K_j)$ of compact sets we provide necessary and sufficient conditions to establish when we have full convergence of the minimisers of \eqref{fct0} to the minimiser of \eqref{eq:1.5} (i.e., without passing to a subsequence). 

\begin{ack}
The authors wish to thank Dorin Bucur for useful discussions about this paper.
This material is based on work supported by the Italian 
Ministry of Education, University, and
Research under the Project ``Calculus of Variations" (PRIN 2015).
All authors are members of the Gruppo Nazionale per l'Analisi Matematica,
la Probabilit\`a e le loro Applicazioni (GNAMPA) of the Istituto Nazionale di Alta Matematica (INdAM),
which provided funding for GF and DZ.
\end{ack}

\section{Technical tools}
\subsection{Assumptions}\label{subsec:2.1}
Throughout the paper, we fix a bounded open set $\Omega\subset\mathbb{R}^n$, $n\ge 2$, and a function $f\colon\mathbb{R}^n\to[0,\infty)$ such that
\begin{subequations}
\label{A}
\begin{align}
& \text{$f(\cdot)$ is convex, even, and positively homogeneous of degree $p$;}
	\label{Ab}\\
 & \lambda |\xi|^p\le f(\xi)\le \Lambda |\xi|^p\,,\ \text{for all $\xi\in\mathbb{R}^n$,}\label{Ac}
\end{align}
\end{subequations}
for suitable constants $1<p\le n$ and $0<\lambda<\Lambda<\infty$.
We also fix a function $g:\Omega\times\mathbb R\to[0,\infty)$ such that
\begin{subequations}\label{C}
\begin{align}
& \text{$g(\cdot,s)$ is measurable for all $s\in\mathbb{R}$;} \label{C1} \\
& \text{$g(x,\cdot)$ is continuous for a.e.\ $x\in\Omega$;} \label{C2}\\
& \text{ $c_1 |s|^q- a_1(x)\le g(x,s)\le c_2(1+|s|^q+a_2(x))$ for a.e.\ $x\in\Omega$ and all $s\in\mathbb R $,} \label{C3}
\end{align}
\end{subequations}
for suitable constants $1\le q<\infty$, $0<c_1<c_2<\infty$, and functions $a_1,a_2\in L^1(\Omega)$.  

Moreover, we fix a compact $(n-1)$-dimensional $C^1$ manifold $M$ with boundary, contained in $\Omega$, and a sequence $(K_j)$ of compact sets of $\mathbb{R}^n$ such that 
\begin{equation}
\label{eq:1.1sec3}
\lim_{j\to\infty}\max_{x\in K_j} {\rm dist}(x,M)=0\,.
\end{equation}
Eventually, $\Sigma$ will denote a fixed compact $(n-1)$-dimensional $C^1$ manifold with boundary, contained in $\Omega$, with 
$M\subset \Sigma\meno\partial\Sigma$.

\subsection{Some fine properties of Sobolev functions}\label{subsec:2.2}
Given an open set $A\subset\mathbb R^n$,
$L^0(A)$ stands for the space of all (equivalence classes of) real valued measurable functions on $A$,
endowed with the topology of the convergence in measure. Note that the topological space $L^0(A)$ is
metrisable and separable. The Sobolev space $W^{1,p}(A)$ consists, as usual, 
of all functions $u\in L^p(A)$ whose distributional gradient $\nabla u$ belongs to $L^p(A;\mathbb R^n)$. We shall also make use
of the Deny-Lions space $L^{1,p}(A)$, i.e., the set of all functions $u\in L^1_{\rm loc}(A)$ whose distributional gradient $\nabla u$ belongs to $L^p(A;\mathbb{R}^n)$. We recall that if $A$ is locally the subgraph of a Lipschitz function near $x\in\partial A$, then there exists an open neighbourhood $A'$ of $x$ such that $L^{1,p}(A\cap A')=W^{1,p}(A\cap A')$. In particular,  $L^{1,p}(A)=W^{1,p}(A)$ whenever $A$ is a
bounded Lipschitz open set. It is also known that $\{\nabla u:\, u\in L^{1,p}(A)\}$ is a closed subspace of $L^p(A;\mathbb R^n)$.
For a more detailed account about these spaces, the reader is referred to \cite{DL} and~\cite{Mazya}. 
Here and henceforth, for every $u,v\in L^0(A)$ we set $(u\wedge v)(x):=\min\{u(x),v(x)\}$ \text{and} $(u\vee v)(x):=\max\{u(x),v(x)\}$.
For every $u\in L^0(\Omega)$ and $t>0$, we denote the truncation  of $u$ at level $t$ by $u^t:=u\wedge t \vee -t$.

The $p$--capacity $C_{1,p}$ of a set $E\subset \Omega$ is defined as 
\[
C_{1,p}(E):=\inf_{u\in \mathcal U(E)}\int_{\Omega} |\nabla u|^p\,dx\,,
\]
where $\mathcal U(E)$ is the set of all $u\in W^{1,p}_0(\Omega)$ such that $u\ge1$
almost everywhere in an open neighbourhood of $E$.
Then, one says that a property holds $p$-quasi everywhere or equivalently for $p$-q.e.\ $x$, both abbreviated to $p$--q.e.,
if the points where it fails form a set of $p$--capacity zero. The usual abbreviation a.e. for almost everywhere, if not specified, always refers to the Lebesgue measure.
Given a set $E\subset \Omega$, a function $u$ defined on $E$ is said to be $p$--quasicontinuous if for every $\e>0$ there exists a set $E'$ with $C_{1,p}(E')<\e$ such that the restriction of $u$ to $E\setminus E'$ is continuous.
The notions of $p$-quasi upper and $p$-quasi lower semicontinuity are defined in a similar way.
A set $U\subset \Omega$ is said to be $p$-quasi open in $\Omega$ if for every $\varepsilon>0$ there exists an open set $A\subset \Omega$ such that $C_{1,p}(U\triangle A)<\varepsilon$, where $\triangle$ denotes the symmetric difference of sets. Given a $p$-quasi open set $U$ in $\Omega$, for every $\varepsilon>0$ there exists an open set $V\subset \Omega$ such that
$U\cup V$ is open and $C_{1,p}(V)<\varepsilon$. The definition of $p$-quasi closed is analogous and we
have that $U$ is $p$-quasi open in $\Omega$ if, and only if, $\Omega\setminus U$ is $p$-quasi closed.
It is easily seen that a set $U\subset\Omega$ is $p$-quasi open ($p$-quasi closed) if and only if its characteristic function $1_U$ is $p$-quasi lower ($p$-quasi upper) semicontinuous.
It can be proved that a function $f\colon \Omega\to [-\infty,\infty]$ is $p$-quasi lower ($p$-quasi
upper) semicontinuous if and only if the sets $\{ x\in\Omega\colon f(x)>t\}$
($\{x\in\Omega\colon f(x)\ge t\}$) are $p$-quasi open ($p$-quasi closed) for all $t\in\mathbb R$.

Given an open set $A\subset \Omega$, for all $u\in W^{1,p}(A)$ there exists a $p$--quasicontinuous function
$\tilde u$ that coincides with $u$ a.e., called the $p$--quasicontinuous representative of $u$; it is well known that $\tilde u$ is uniquely determined $p$-q.e.\ and that
\begin{equation}\label{defmed2.1}
\lim_{\rho\to0^+}\intmed_{B_{\rho}(x)}|u(y)-\tilde u(x)|\,dy=0\,,\qquad \text{for $p$-q.e.\ $x\in A$.}
\end{equation}
For a complete treatment of the notion of capacity and of the fine properties of Sobolev functions, we refer to
the books~\cite{EG,HKM,Mazya,Z}.

Incidentally, under suitable assumptions
the values of Sobolev functions can be made precise also at $p$-q.e.\ boundary point.
More precisely, if $A\subset\subset \Omega$ is an open set with Lipschitz boundary and $u\in W^{1,p}(A)$ then
there exists a $p$-quasicontinuous function on $\overline A$, which we still denote by $\tilde u$,
such that
\begin{equation}
\label{rapprqcont}
	\lim_{\rho\to0^+} \frac{1}{\rho^n}\int_{A\cap B_\rho(x)}|u(y)-\tilde u(x)|\,dy=0\,,
	\qquad \text{for $p$--q.e.\ $x\in\overline A$.}
\end{equation}
Indeed, the extension theory for Sobolev spaces implies that there exists $v\in W^{1,p}(\Omega)$
such that $v=u$ a.e. in $A$. Applying \eqref{defmed2.1} to the $p$-quasi continuous representative $\tilde v$
of $v$ we obtain that
\[
\limsup_{\rho\to0^+} \frac{1}{\rho^n}\int_{A\cap B_\rho(x)} |u(y)-\tilde v(x)|\,dy \le
 \lim_{\rho\to0^+} \frac{1}{\rho^n}\int_{B_\rho(x)}|v(y)-\tilde v(x)|\,dx=0\,,
\]
for $p$-q.e.\ $x\in\overline A$.
It is now enough to define $\tilde u$ as the restriction of $\tilde v$ to $\overline A$.
On the other hand, it is clear that \eqref{rapprqcont} uniquely determines $\tilde u$ $p$-q.e.\ on $\overline A$.

By standard properties of the traces of Sobolev functions (see, e.g.,~\cite[Theorem 3.87]{AFP},
\cite[Theorem 2, p. 181]{EG}),
\eqref{rapprqcont} implies that $\tilde u_{|_{\partial A}}$
coincides $\mathscr{H}^{n-1}$--a.e. with the trace $\gamma(u)$ of $u$
on $\partial A$, where $\mathscr{H}^{n-1}$ denotes the $(n-1)$-dimensional Hausdorff measure.
Moreover, $\tilde u_{|_{\partial A}}$ is $p$-quasicontinuous.

Since $\gamma(u)\in W^{1-1/p,p}(\partial A)$, it is possible to prove that
$\tilde u_{|_{\partial A}}$ coincides with the quasicontinuous representative of $\gamma(u)$ with respect to the
fractional capacity $C_{1-1/p,p}$, for which we refer to \cite[Section 10.4]{Mazya} or \cite[Section 2.6]{Z}. However, this last 
property will never be used in this paper.


From what noticed above a set $U$ is $p$-quasi open if, and only if, there exists $u\in W^{1,p}(\mathbb R^n)$  with $U=\{x\in\mathbb R^n\colon u(x)>0\}$. Moreover, we have the following result.

\begin{lm}\label{lm:qa0}
Let $U$ be a $p$-quasi open set in $\Omega$. Then there exists a sequence of compact sets
$(Q_k)$ with $Q_k\subset U$ and a monotonically non-decreasing sequence of functions $(\chi_k)$ with $\chi_k\in W^{1,p}_0(\Omega)$, with $0\le \chi_k\le1$, $\chi_k=0$\ in $\Omega\setminus Q_k$, and $\chi_k\to1$ $p$-q.e.\ in $U$.
\end{lm}
\begin{proof}
We take $u\in W^{1,p}(\mathbb R^n)$ with $U=\{u>0\}$.
For every $k>0$ we define $u_k\in W^{1,p}(\mathbb R^n)$ setting $u_k = \big((u-\tfrac{1}{k})\vee 0)\cdot k\big)\wedge1$.
Clearly $0\le u_k\le1$ and out of the $p$-quasi closed set $U_k:=\{u\ge\tfrac{1}{k}\}$ we have that $u_k=0$. For every $k$, let $v_k\in W^{1,p}_0(\Omega)$, with $0\le v_k\le1$, be such that
\[
	\int_\Omega |\nabla v_k|^p\,dx = C_{1,p}(V_k)\,,
\]
where $V_k$ is an open set,
with $C_{1,p}(V_k)<\frac{1}{k}$, such that $U_k\setminus V_k$ is
closed. It follows that the function $w_k:=u_k\wedge(1-v_k)$ belongs to $W^{1,p}(\Omega)$ and
vanishes out of the compact set $Q_k := U_k\setminus V_k$. Moreover, $w_k$ converges to $1$
$p$-q.e.\ in $U$, since so does $u_k$ by construction. Therefore,
the sequence defined by recursion setting $\chi_1:=w_1$, and $\chi_k:=w_k\vee \chi_{k-1}$ for all $k>1$,
is a monotonically non-decreasing sequence that satisfies the desired properties.
\end{proof}

We introduce a $p$--quasicontinuous version of the traces of a piecewise $W^{1,p}$ function on both sides
of a manifold. We begin with the case of a graph of a $C^1$ function.
Given $x_0\in \mathbb{R}^n$,
$\nu_0\in \mathbb{R}^n$ with $|\nu_0|=1$, and $r_0>0$, we consider the cylinder defined by 
\begin{subequations}\label{notazioni}
\begin{equation}
\label{cilynder}
\Omega_0:=\left\{ x\in\mathbb{R}^n\colon |(x-x_0)\cdot\nu_0|<r_0\,, \ |x-x_0 - ((x-x_0)\cdot\nu_0)\nu_0|<r_0\right\}\,.
\end{equation}
We fix a function $\phi$ of class $C^1$  defined on the $(n-1)$-dimensional disk $\Pi_0= \{x\in \Omega_0\colon (x-x_0)\cdot \nu_0=0\}$ with values in $\left(-\tfrac{r_0}{4},\tfrac{r_0}{4}\right)$
and vanishing at $x_0$. Its graph is defined by
\begin{equation}
\label{Ngraph}
	\Sigma_0:=\left\{\bar x+ \phi(\bar x)\nu_0\colon \bar x\in \Pi_0\right\}\,.
\end{equation}
We now define the open sets $\Omega_0^+$ and $\Omega_0^-$ by
\begin{equation}
\label{2Omega}
\Omega_0^\pm = \left\{\bar x+r\nu_0\colon \bar x\in \Pi_0\,, \ r\in \left(-r_0,r_0\right) \,, \ \pm (r- \phi(\bar x))>0\right\}\,,
\end{equation}
\end{subequations}
so that $\Omega_0=\Omega_0^+\cup \Sigma_0\cup \Omega_0^-$.
\begin{lm}\label{lm:2.1}
Let $\Omega_0$, $\Sigma_0$, and $\Omega_0^\pm$ be defined by \eqref{notazioni}, and
let $\nu$ be the continuous unit normal
to $\Sigma_0$ such that $\nu(x_0)=\nu_0$.
Then for every $u \in W^{1,p}(\Omega_0\setminus \Sigma_0)$ there exist two $p$-quasicontinuous functions $u^+ $ and $u^-$ defined
on $\Sigma_0$ such that
\begin{equation}\label{intmedbd}
	\lim_{\rho\to0^\pm}\intmed_{B_{\rho,\nu}^\pm(x)}|u(y)-u^\pm(x)|\,dy=0\,,
		\qquad \text{for $p$-q.e.\ $x\in \Sigma_0$,}
\end{equation}
where
\(
	B_{\rho,\nu}^\pm(x) := \left\{y\in B_\rho(x)\,\,\colon \pm(y-x)\cdot\nu(x)>0 \right\}
	\).
\end{lm}
\begin{proof}
Let $u\in W^{1,p}(\Omega_0\meno \Sigma_0)$ be fixed. The extension theorems for Sobolev functions
imply that that there exist $v,w\in W^{1,p}(\Omega_0)$ such that
$v=u$\ a.e.\ in $\Omega_0^+$ and $w=u$\ a.e.\ in $\Omega_0^-$.
Let us set $u^+(x):=\tilde v(x)$ and $u^-(x):=\tilde w(x)$ for all $x\in \Sigma_0$.
Then for $p$-q.e.\ $x\in \Sigma_0$ we have
\[
\int_{\Omega_0^+\cap B_{\rho}(x)} \!|u(y)-u^+(x)|\,dy
=\int_{\Omega_0^+\cap B_{\rho}(x)} \!|v(y)- \tilde v(x)|\,dy
\le \int_{ B_{\rho}(x)} \!|v(y)- \tilde v(x)|\,dy 
\]
and a similar inequality holds when $u^+$ and $v$ are replaced by $u^-$ and $w$.
Therefore, by~\eqref{defmed2.1}
\[
\lim_{\rho\to0^+}\intmed_{\Omega_0^\pm\cap B_{\rho}(x)} |u(y)-u^\pm(x)|\,dy=0\,.
\]
It is easily seen that
$\Omega_0^\pm\cap B_{\rho}(x)$ can be replaced by $B_{\rho,\nu}^\pm(x)$ (see, e.g.,~\cite[Corollary 1, p. 203]{EG}) and by this we can conclude.
\end{proof}

\begin{rmq}
Note that under the assumptions of Lemma~\ref{lm:2.1} the functions $u^\pm$
coincide $\mathscr{H}^{n-1}$--a.e. on $\Omega_0\cap \Sigma_0$ with
the traces of $u_{\vert_{\Omega_0^+}}\in W^{1,p}(\Omega_0^+)$ and $u_{\vert_{\Omega_0^-}}\in W^{1,p}(\Omega_0^-)$ in the sense of Sobolev spaces.

Moreover, since $\Omega_0^+$ and $\Omega_0^-$ have Lipschitz boundary 
and $\Sigma_0\subset \partial \Omega_0^+\cap \partial \Omega_0^-$, 
by \eqref{rapprqcont} the trace $u^\pm$ concides $p$-q.e.\ on $\Sigma_0$ with the restriction to $\Sigma_0$ of
the $p$-quasicontinuous representative of $u_{\vert_{\Omega_0^+}}$ on $\overline{\Omega_0^+}$
and of $u_{\vert_{\Omega_0^-}}$ on $\overline{\Omega_0^-}$.
\end{rmq}

In the following lemma we introduce the absolute value of the jump of a function across a manifold.

\begin{lm}\label{lm:2.2}
Let $A$ be an open subset of $\Omega$. 
Then for every $u\in L^{1,p}(A\meno \Sigma)$ there exists a $p$-quasicontinuous function $[u]$ on $A\cap \Sigma$ such that
for $p$-q.e.\ $x\in A\cap \Sigma$
\begin{equation}
\label{jumpdef}
[u](x)=\bigg\vert \lim_{\rho\to0^+}\intmed_{B_{\rho,\nu}^+(x)}u(y)\,dy-\lim_{\rho\to0^+}\intmed_{B_{\rho,\nu}^-(x)}u(y)\,dy\bigg\vert\,,
\end{equation}
where $\nu$ is a unit normal to $\Sigma$ at $x$. 
\end{lm}
\begin{proof}
For every point $x_0\in A\cap \Sigma$ there exists $r_0>0$ such that 
the intersection of $\Sigma$ with the cylinder $\Omega_0$ defined in~\eqref{cilynder} is
the graph of a $C^1$ function as in \eqref{Ngraph}.
Therefore, by Lindel\"of  Theorem $A\cap \Sigma$ can be covered by a sequence $(\Omega_i)$ of such cylinders, each well contained in $A$ so that $u\in W^{1,p}(\Omega_i\setminus \Sigma)$.

By Lemma~\ref{lm:2.1}, for every $i$ there exist two $p$-quasicontinuous functions $u^+_i$ and $u^-_i$
on $\Omega_i\cap \Sigma$ such that~\eqref{intmedbd} holds. Note that for $p$-q.e.\ $x\in  
\Omega_i\cap \Omega_j\cap \Sigma$
we have either
\(
u^\pm_i(x) = u^\pm_{j}(x)
\)
or
\(
u^\pm_i(x) = u^\mp_{j}(x)
\),
which implies that
$|u^+_i(x)-u^-_i(x)|= |u^+_j(x)-u^-_j(x)|$. 
This allows us to define
\[
	[u](x) := |u^+_{i}(x)-u^-_{i}(x)|\,,\qquad \text{for $p$-q.e.\ $x\in \Omega_i\cap\Sigma$.}
\]
It is then immediate to see that $[u]$ is $p$-quasicontinuous and that \eqref{jumpdef} holds.
\end{proof}
\begin{rmq}\label{convpqe} 
Since the truncation $u\mapsto u^t=u\wedge t \vee -t$ is Lipschitz continuous, it commutes with the choice of a $p$-quasi continuous representative. Hence by Lemma~\ref{lm:2.2} we can infer that $[u^t]$ converges to $[u]$, as $t\to \infty$,  $p$-q.e. on $A\cap \Sigma$.
\end{rmq}

\subsection{Capacitary measures}\label{subsec:2.3}
In this subsection we introduce the class of measures which are involved in the integral representation of the limit problem \eqref{eq:1.5}.
In the sequel, for every open set $A\subset\mathbb{R}^n$
the symbol $\mathscr{A}(A)$ denotes the collection of all open
subsets of $A$, and for every Borel set $B\subset \mathbb{R}^n$
by $\mathscr{B}(B)$ we denote the collection 
of all Borel subsets of $B$. 

\begin{dfn}
Given $A\in \mathscr{A}(\Omega)$, the symbol $\mathcal{M}_p(A)$ 
denotes the class of
measures $\mu\colon \mathscr{B}(A)\to[0,\infty]$
such that $\mu(B)=0$ for every $B\in\mathscr{B}(A)$ with $C_{1,p}(B)=0$.
In addition, $\mathcal{M}_p(A;M)$ stands for the subclass of measures $\mu\in \mathcal{M}_p(A)$ with $\mu(A\setminus M)=0$.
\end{dfn}

Let $A\in \mathscr A(\Omega)$. Since $\Omega$ is bounded, a Radon measure $\mu$ on $\Omega$, considered as a distribution on $A$, belongs to the dual space
$W^{-1,p'}(A)$ of $W^{1,p}_0(A)$ if and only if there exists
$\Psi\in L^{p'}(A;\mathbb R^n)$ with $\mu = \mathop{\rm div} \Psi$ in the sense of distributions, {\em i.e.},
\begin{equation*}
	\int_A \varphi\,d\mu =- \int_A \nabla \varphi\cdot \Psi\,dx\,,
\end{equation*}
for all $\varphi\in C^{\infty}_c(A)$. We will denote by $\langle\,\cdot\,,\cdot\,\rangle$ the duality pairing between
$W^{-1,p'}(A)$ and $W^{1,p}_0(A)$ (see~\cite{HKM,Z}).

By this integral representation and H\"older inequality, all non-negative Radon measures on $A$ of class $W^{-1,p'}(A)$ belong to $\mathcal{M}_p(A)$.
In particular, 
a Radon measure $\mu$ on $A$ of class $W^{-1,p'}(A)$ that is supported on $M$
belongs to
$\mathcal{M}_p(A;M)$. 
Another significant istance within this class is the Hausdorff measure $\mathscr H^{n-1}\mres(A\cap M)$;
more in general, the Radon measure $\mathscr H^s\mres E$ belongs to $\mathcal{M}_p(A)$ whenever
$E$ is a subset of $A$ such that $\mathscr H^{s}(E)<\infty$ with $s>n-p$.
Nevertheless, measures in $\mathcal{M}_p(A;M)$
are required neither to be inner regular not to be locally finite. For example
\begin{equation}
\label{misura:infinita}
	\infty_E(B) := \begin{cases}
	0\,, & \qquad \text{if $C_{1,p}(B\cap E)=0$\,,}\\ \infty\,, & \qquad \text{otherwise,}
	\end{cases}
	\qquad \text{for all $B\in\mathscr B(A)$,}
\end{equation}
defines a measure of class $\mathcal{M}_p(A)$ for all $E\subset \Omega$, and clearly $\infty_E\in\mathcal{M}_p(
 A;M)$ if $E\subset M$.
\begin{dfn}
\label{dfn:equiv}
We say that two measures $\mu_1,\mu_2\in\mathcal{M}_p(\Omega; M)$ are equivalent if 
\begin{equation}
\label{equivmis}
\int_{ M} [u]^p\,d\mu_1=\int_{ M}[u]^p\,d\mu_2
\end{equation}
for all $u\in L^{1,p}(\Omega\setminus M)$.
\end{dfn}

It is easy to see that Definition \ref{dfn:equiv} implies that
\begin{equation}\label{equivmisA}
	\int_{A\cap M}[u]^p\,d\mu_1=\int_{A\cap M}[u]^p\,d\mu_2\,,
\end{equation}
for every $A\in\mathscr A(\Omega)$ and for every $u\in L^{1,p}(A\setminus M)$. In fact, for every 
$\varphi\in C^\infty_0(A)$, $\varphi u$ can be considered as an element of $L^{1,p}(\Omega\setminus  M)$
for which \eqref{equivmis} holds. To obtain \eqref{equivmisA} it is enough to approximate $1_A$ with an increasing sequence
of functions in $C^\infty_0(A)$.

In particular, two equivalent measures must agree on all open sets.
We point out that this necessary condition does not imply, in general, that they agree on all Borel sets (see Example~\ref{exa2.1}),
unless at least one of them is a Radon measure; on the other hand, the coincidence on open sets does not imply the equivalence according to Definition~\ref{equivmis}
(see Example~\ref{exa2.2}).
\begin{example}\label{exa2.1} Let $\mu_1=\infty_E$ be the measure defined in \eqref{misura:infinita} where $E=A\cap  M$,
with $A\in\mathscr A(\Omega)$ fixed, and we define
\[
\mu_2(B) = \begin{cases}
0\,, & \qquad \text{if $\mathscr H^{n-1}(B\cap E)=0$,} \\ \infty \,, & \qquad \text{otherwise.}
\end{cases}
\qquad \text{for all $B\in\mathscr B(A)$,}
\]
Then $\mu_1$ and $\mu_2$ are different but equivalent, because for a $p$-quasi continuous function $v$
we have that $v=0$ holds $\mathscr H^{n-1}$-a.e.\ on $A\cap  M$ if and only if it holds
$p$-q.e.\ on $A\cap  M$. 
\end{example}

\begin{example}\label{exa2.2}
In this example, $\mu_1=\infty_{E_1}$ and $\mu_2=\infty_{E_2}$ for suitable sets $E_2\subset E_1\subset  M$.
We first construct the set $E_1$.
Let $x_0\in M\setminus\partial M$ and $r_0>0$ be so small that $\Sigma_0=M\cap \Omega_0$
admits a representation of the form \eqref{Ngraph} where $\Omega_0$ is the cylinder defined by \eqref{cilynder}. Let also $\Omega_0^\pm$ be as in \eqref{2Omega}. 
For every $E\subset\Sigma_0$,
we set
\begin{equation}
\label{kappa1}
\kappa(E) := \inf_{u\in \mathcal{U}(E)}  \int_{\Omega_0^+}|\nabla u|^p\,dx
\end{equation}
where $\mathcal{U}(E)$ is the set of all functions $u\in W^{1,p}(\Omega_0^+)$ such that
$u=0$  $p$-q.e. on $E$ and $u=1$ $p$-q.e. on $\partial\Omega_0^+\setminus M$.
We fix $r_1\in(0,r_0)$, we consider the cylinder $\Omega_1$ defined as in
\eqref{cilynder} with $r_0$ replaced by $r_1$, and we set $E_1:=M \cap \Omega_1$. 

To construct $E_2$, we fix a sequence $(x_i)$ of points dense in $E_1$ and 
\[
E_2 := \bigcup_{i\in\mathbb N} \big(B_{\rho_i}(x_i)\cap  M\,\big)
\]
for a suitable choice of the radii $\rho_i>0$. The first condition on $\rho_i$ is that $B_{\rho_i}(x_i)\cap M\subset E_1 $.
The second one is the inequality $\kappa(E_1)>\kappa(E_2)$, which can be obtained using the countable subadditivity of
$\kappa(\cdot)$ and the fact that
\[
	\lim_{\rho\to0^+} \kappa(B_{\rho}(x))=0\qquad \text{for every $x$.}
\]
These two properties of $\kappa(\cdot)$ can be proved with the same arguments used for $C_{1,p}$.

Since $E_1$ and $E_2$ are relatively open in $ M$ and $E_2$ is dense in $E_1$, for all $A\in\mathscr A(\Omega)$ we have
\[
C_{1,p}(E_1\cap A)>0 \quad \Longleftrightarrow \quad C_{1,p}(E_2\cap A)>0\,.
\]
This implies that $\mu_1(A)=\mu_2(A)$ for every $A\in\mathscr A(\Omega)$.

To prove that $\mu_1$ and $\mu_2$ are not equivalent, we observe that the inequality $\kappa(E_1)>\kappa(E_2)$ implies the existence of a function $u\in W^{1,p}(\Omega_0^+)$, with $u=0$ $p$-q.e.\ on $E_2$ and $u=1$
$p$-q.e.\ on $\partial\Omega_0^+\setminus M$, such that
\begin{equation}
\label{exa:contr}
	\int_{\Omega_0^+} |\nabla u|^p\,dx< \kappa(E_1)\,.
\end{equation}
Let $v\in W^{1,p}(\Omega\setminus  M)$ be such that $v=u$ in $\Omega_0^+$ and $v=0$ in $\Omega_0^-$.
We claim that
\begin{equation}
\label{exa:eq}
\int_{\Omega}[v]^p\,d\mu_1=\infty\quad \text{and}\quad \int_{\Omega}[v]^p\,d\mu_2=0\,.
\end{equation}
To prove the first equality, by contradiction we assume that the first integral is finite. By the definition of $\mu_1$
this implies that $[v]=0$, hence $u=0$, $p$-q.e.\ in $E_1$. Thus $u$ is a competitor for the minimum problem \eqref{kappa1}
which defines $\kappa(E_1)$, contradicting \eqref{exa:contr}. The second equality in \eqref{exa:eq} 
follows from the fact that $[v]=u=0$ $p$-q.e.\ on $E_2$.
\end{example}

\begin{lm}\label{lm:criterionequiv}
Two measures $\mu_1$, $\mu_2\in \mathcal{M}_p(\Omega; M)$ are equivalent if and only if
they agree on $p$-quasi open sets.
\end{lm}
\begin{proof}
The criterion follows by repeating with obvious changes the arguments used in~\cite[Theorem 2.6]{DM2}.
\end{proof}

The following lemma introduces a distinguished element in each equivalence class of $\mathcal{M}_p(\Omega; M)$.

\begin{lm}
For every $\mu\in \mathcal{M}_p(\Omega; M)$, there exists a measure $\mu^\ast\in\mathcal{M}_p(\Omega; M)$, equivalent to $\mu$, with the property that
\begin{equation}
\label{maxcoset}
	\mu^\ast(B) = \inf\big\{ \mu^\ast(U)\colon \text{\em $U$ $p$-quasi open and $B\subset U\subset \Omega$}\big\}\,,
\end{equation}
for all $B\in\mathscr B(\Omega)$. Moreover, $\mu^\ast\ge\nu$ whenever $\nu\in\mathcal{M}_p(\Omega; M)$ is equivalent to $\mu$.
\end{lm}
\begin{proof}Arguing as in~\cite[Theorem 3.9]{DM2}, 
it can be seen that
\begin{equation*}
\mu^\ast(B) := \inf\big\{ \mu(U)\colon\text{$U$ $p$-quasi open and $B\subset U\subset \Omega$}\big\}\,,\qquad \text{for all $B\in\mathscr B(\Omega)$,}
\end{equation*}
defines a Borel measure of class $ \mathcal{M}_p(\Omega; M)$ satisfying \eqref{maxcoset}. By Lemma~\ref{lm:criterionequiv}
$\mu^\ast$ is equivalent to $\mu$ and for any other measure $\nu$ within the equivalence class
we have that
\[
\nu(B) \le \inf\big\{\nu(U)\colon\text{$U$ $p$-quasi open and $B\subset U\subset \Omega$}\big\}
=\mu^\ast(B)\,, 
\]
for all $B\in\mathscr B(\Omega)$.
\end{proof}

The previous lemma states that any capacitary measure admits a maximal representative in its equivalence class
that is outer regular with respect to $p$-quasi open sets.

In the sequel we will sometimes need to represent
$\mu$ by means of a measure which is absolutely continuous with respect to an element of a dual Sobolev space.

\begin{lm}\label{lm:equivmis}
Every measure $\mu\in \mathcal{M}_p(\Omega;M)$ 
is equivalent to a measure $\psi\sigma$ of the form
\[
	\psi \sigma(B) = \int_{B\cap M}\psi\,d\sigma\,,\qquad B\in\mathscr B(\Omega)\,,
\]
for a Borel function $\psi:M\to[0,\infty]$ and a non-negative measure $\sigma$ of class $W^{-1,p'}(\Omega)$.
\end{lm}
\begin{proof}
Let $\mu\in \mathcal{M}_p(\Omega;M)$. 
Let $\Omega_0$, $\Sigma_0$, and $\Omega_0^\pm$ be as in \eqref{notazioni}, let $E_0$ be a Borel set of $\Sigma_0$ and
let
$\mu_0=\mu\mres E_0\in \mathcal M_p(\Omega_0,\Sigma_0)$. 
By~\cite[Theorem 5.7]{DM}, there exist a Borel function $h\colon \Omega_0\times\mathbb R\to[0,\infty]$, increasing and lower semicontinuous in the second variable for all $x\in\Omega_0$, and a non-negative Radon measure $\sigma_0$ of class $W^{-1,p'}(\mathbb R^n)$ such that
\begin{equation}\label{eq:lemma}
\int_{A\cap \Sigma_0}
(\tilde z\vee0)^p
d\mu_0=
\int_{A\cap \Sigma_0} h(x,\tilde z)\,d\sigma_0\,,
\end{equation}
for all $A\in\mathscr A(\Omega_0)$ and $z\in W^{1,p}(\Omega_0)$. The additional non-negative Borel measure $\nu$ appearing in~\cite[Theorem 5.7]{DM} is not present here due to the obvious $p$-homogeneity property of the integral in the left-hand side of \eqref{eq:lemma}. Clearly, by \eqref{eq:lemma}, for every $t>0$ we have
\[
\int_{A\cap \Sigma_0} h(x,t\tilde  z)\,d\sigma_0= t^p\int_{A\cap \Sigma_0} h(x,\tilde z)\,d\sigma_0\,.
\]
Then by~\cite[Lemma 2.3]{DMD}, setting $\psi_0(x):=h(x,1)$ for $x\in \Omega_0$, we deduce 
\[
	\int_{A\cap \Sigma_0} h(x,\tilde z)\,d\sigma_0=\int_{A\cap \Sigma_0}(\tilde z\vee0)^p \psi_0\,d\sigma_0\,.
\]
Plugging this into \eqref{eq:lemma} we obtain that
\begin{equation}
\label{equivloc}
	\int_{A\cap \Sigma_0} \tilde z^p\,d\mu_0=\int_{A\cap \Sigma_0} \tilde z^p\psi_0 \,d\sigma_0\,,
\end{equation}
for all $A\in\mathscr A(\Omega_0)$ and for all non-negative functions $z\in W^{1,p}(\Omega_0)$.

We now prove the equivalence of $\mu$ to some measure $\psi\sigma$ according to Definition~\ref{dfn:equiv}.
Let  $\{E_i\}$ be a finite family of pairwise disjoint Borel sets of $M$ such that
 $M = E_1\cup\ldots\cup E_m$ for some $m\in\mathbb N$, every $E_i$ is contained in a 
cylinder  $\Omega_i$ of the form~\eqref{cilynder} and $\Sigma_i=\Sigma\cap \Omega_i$ admits a graph representation as in \eqref{Ngraph}, for suitable radii $r_i>0$, centres $x_i\in M$, and axis $\nu_i\in \mathbb{R}^n$ with $|\nu_i|=1$.
Let $\psi_i:\Sigma_i\to[0,\infty]$ and $\sigma_i\in W^{-1,p'}(\Omega_i)$ 
satisfy \eqref{equivloc} with $\Omega_0$, $\Sigma_0$, $\mu_0$ replaced by 
$\Omega_i$, $\Sigma_i$, $\mu_i=\mu\mres E_i$. Then we fix a function $u\in L^{1,p}(A\setminus M)$. For all indices $i$, 
the extension theorems for Sobolev functions imply
that there exist $v_i,w_i\in W^{1,p}(\mathbb R^n)$
such that $v_i=u$ a.e.\ in $\Omega_i^+$ and $w_i=u$ a.e.\ in $\Omega_i^-$,
with $\Omega_i^\pm$ being defined as in \eqref{2Omega}. Let
$z_i = \tilde v_i-\tilde w_i$ so that $[u]=|\tilde z_i|$ for $p$-q.e. $x\in \Sigma_i$ and notice that $z_i=0$ on $\Sigma_i\setminus M$.
Hence by \eqref{equivloc} we have
\begin{equation*}
\int_{A\cap M} [u]^p\,d\mu=
\sum_{i=1}^m\int_{A\cap \Sigma_i}  |z_i|^p\,d\mu_i =
\sum_{i=1}^m\int_{A\cap \Sigma_i} |z_i|^p\,\psi_i\,d\sigma_i= \int_{A\cap M}[u]^p\psi\,d\sigma\,,
\end{equation*}
where in the last equality we have set
\begin{equation*}
\sigma = \sum_{i=1}^m\sigma_i\,,\qquad \psi = \sum_{i=1}^m 1_{E_i}\psi_i\,.
\end{equation*}
By linearity, it is clear that $\sigma$ is a non-negative Radon measure of class $W^{-1,p'}(\mathbb R^n)$ and that
$\psi$ is a non-negative Borel function on $M$. The lemma is then proved.
\end{proof}

\section{$\Gamma$-convergence}\label{sec:3}

We prove in this section a compactness result about the $\Gamma$-convergence of a sequence of functionals involving the compact sets $K_j$ and the function $f$. Then we discuss the convergence of the minimisers of this sequence when it is perturbed by a functional involving the function $g$.

\subsection{A compactness result}
For every $j\in\mathbb N$ let $\mathscr F_{K_j}\colon L^0(\Omega)\times\mathscr{A}(\Omega)\to[0,\infty]$ be the functional defined by
\begin{equation}
\label{eq:1.2}
\mathscr F_{K_j}(u,A) := \begin{cases}
			\displaystyle\int_{A\meno  K_j} f(\nabla u)\,dx\,, & \qquad\text{if $u_{\vert_{{A\meno  K_j}}}\in L^{1,p}(A\meno  K_j)$,}\\
			\infty\,,& \qquad \text{otherwise,}
\end{cases}
\end{equation}
with $f$ satisfying \eqref{A} and $(K_j)$ as in \eqref{eq:1.1sec3}.
Notice that for every $u\in L^0(\Omega)$ the set functions $\mathscr F_{K_j}(u,\cdot)$ are
increasing on $\mathscr{A}(\Omega)$ with respect to set inclusion. 
 
Then, let $\mathscr F',\mathscr F''\colon L^0(\Omega)\times \mathscr{A}(\Omega)\to[0,\infty]$ be the functionals defined by
\begin{subequations}\label{233}
\begin{align}
 \mathscr F'(u,A) & := \inf\left\{ \liminf_{j\to\infty} \mathscr F_{K_j}(u_j,A)\,\,\colon\,\, \text{$u_j\to u$ in $L^0(\Omega)$}\right\}\,,\label{233a}\\
 \mathscr F''(u,A) &:= \inf\left\{ \Limsup_{j\to\infty} \mathscr F_{K_j}(u_j,A)\,\,\colon\,\, \text{$u_j\to u$ in $L^0(\Omega)$}\right\}\,.\label{233b}
\end{align}
\end{subequations}
We observe that the infima in \eqref{233} are achieved, see~\cite[Proposition 8.1]{DMb}. Moreover, for every $u\in L^0(\Omega)$ the set functions $\mathscr F'(u,\cdot)$ and $\mathscr F''(u,\cdot)$ are
increasing on $\mathscr{A}(\Omega)$ with respect to set inclusion and for every $A\in \mathscr{A}(\Omega)$
the functionals $\mathscr F'(\cdot,A)$ and $\mathscr F''(\cdot,A)$ are lower semicontinuous on $L^0(\Omega)$.

\begin{rmq}
\label{rmq:trunc}
Let $u\in L^0(\Omega)$ and let $(u_j)$ converge to $u$ in $L^0(\Omega)$. For every $t>0$ let $u_j^t=u_j\wedge t\vee -t$ and $u^t=u\wedge t\vee -t$.
Since the truncation is Lipschitz continuous, $(u_j^t)$ converges to $u^t$ in $L^0(\Omega)$; moreover,
it is easily seen that
$\mathscr F_{K_j}(u_j^t,A)\le \mathscr F_{K_j}(u_j,A)$ for all $A\in \mathscr A(\Omega)$.
Since $u^t=u$ for all $t\ge\|u\|_\infty$, we deduce that for functions $u\in L^\infty(\Omega)$
it is not restrictive to assume the recovery sequences in \eqref{233} to be uniformly bounded by $\|u\|_\infty$ in $L^\infty(\Omega)$. 
Similarly, we also deduce that
\begin{equation}\label{truncation}
\mathscr F'(u^t,A)\le \mathscr F'(u,A)\quad  \text{and}\quad  \mathscr F''(u^t,A)\le \mathscr F''(u,A)\,,
\end{equation}
for every $A\in\mathscr A(\Omega)$. Combining \eqref{truncation} with the lower semicontinuity of $\mathscr F'$ and $\mathscr F''$
with respect to the convergence of $(u^t)$ to $u$ in $L^0(\Omega)$, we get
\begin{equation}\label{limits}
\lim_{t\to\infty}\mathscr F'(u^t,A)=\mathscr F'(u,A)\,,\quad \text{and}\quad \lim_{t\to\infty}\mathscr F''(u^t,A)=\mathscr F''(u,A)\,,
\end{equation}
for every $A\in\mathscr A(\Omega)$. 
\end{rmq}

\begin{rmq}\label{richrmq}
Recall that a family of open sets $\mathscr{R}\subset\mathscr{A}(\Omega)$ is said to be {\em rich} if for every $\{A_t\}_{t\in\mathbb{R}}\subset \mathscr{A}(\Omega)$
such that $A_s\subset\subset A_t$ whenever $s<t$, the set $\{t\in\mathbb{R}\,\,\colon\,\, A_t\not\in \mathscr{R}\}$ is at most countable. In~\cite[Theorem 15.18]{DMb} it is proved that 
\begin{equation*}
	\sup_{\substack{A'\in \mathscr{A}(\Omega) \\ A' \subset\subset A }}\mathscr F'(u,A') = \sup_{\substack{A'\in \mathscr{A}(\Omega) \\ A' 		\subset\subset A }}\mathscr F''(u,A')\,,
\end{equation*}
for every $u\in L^0(\Omega)$ and every $A\in \mathscr{A}(\Omega)$ if and only if
there exists a rich family $\mathscr{R}\subset\mathscr{A}(\Omega)$ such that
\begin{equation*}
		 \mathscr F'(u,A) =\mathscr  F''(u,A)\,,
\end{equation*}
for every $u\in L^0(\Omega)$ and every $A\in\mathscr{R}$. Therefore, from \cite[Theorem 16.9]{DMb}, the sequential characterisation of the $\Gamma$-convergence and definitions \eqref{233}, 
there exists a subsequence of $(K_j)$, not relabelled, and a rich family $\mathscr R\subset \mathscr A(\Omega)$ such that
the sequence of functionals $(\mathscr F_{K_j}(\cdot, A))$ $\Gamma$-converges in $L^0(\Omega)$ whenever $A\in\mathscr R$.
\end{rmq}

For every $\mu\in \mathcal M_p(\Omega;M)$ let $\mathscr F^{\mu}\colon L^0(\Omega)\times\mathscr{A}(\Omega)\to[0,\infty]$ be the functional defined by
\begin{equation}
\label{FmuA}
\mathscr F^\mu(u,A) =
 \begin{cases}
			\displaystyle\int_{A\setminus M} f(\nabla u)\,dx+ \int_{A\cap M}[u]^p\,d\mu\,, & \qquad\text{if $u_{\vert_{A\meno M}}\in L^{1,p}({A\meno M})$,}\\
			\infty\,,& \qquad \text{otherwise.}
\end{cases}
\end{equation}

\begin{lm}\label{lm:p0}
Let $\mu\in \mathcal M_p(\Omega;M)$ and let $A\in\mathscr{A}(\Omega)$. Then,
the restrictions to $L^q(\Omega)$ of the functionals $(\mathscr F_{K_j}(\cdot,A))$ defined by \eqref{eq:1.2} $\Gamma$-converge in $L^q(\Omega)$ to the restriction to $L^q(\Omega)$ of $\mathscr{F}^\mu(\cdot,A)$ if and only if $(\mathscr F_{K_j}(\cdot,A))$ $\Gamma$-converge in $L^0(\Omega)$ to  $\mathscr{F}^\mu(\cdot,A)$.
\end{lm}
\begin{proof}
Assume that the restrictions $\Gamma$-converge in $L^q(\Omega)$ and let us prove the $\Gamma$-convergence in $L^0(\Omega)$. We fix $u\in L^0(\Omega)$.
As for the $\Gamma$-limsup inequality, we have to prove that 
\begin{equation}\label{limsupLq2}
\mathscr F^\mu(u,A)\ge \mathscr F''(u,A)\,.
\end{equation}
Let $t>0$. Since the truncated function $u^t\in L^q(\Omega)$ by assumption
there exists a sequence $u_j\in L^{q}(\Omega)$
converging  to $u^t $  in $L^q(\Omega)$ and such that $\mathscr F^\mu(u^t,A)\ge \Limsup_j \mathscr F_{K_j}(u_j, A)$.
Moreover, since convergence in $L^q(\Omega)$ implies convergence in $L^0(\Omega)$,
we have $ \Limsup_j \mathscr F_{K_j}(u_j, A)\ge \mathscr F''(u^t,A) $.
 By \eqref{FmuA}, we have $\mathscr F^\mu(u,A)\ge \mathscr F^\mu(u^t,A)$ and taking the limit as $t\to\infty$ we obtain
 \eqref{limsupLq2}, thanks to \eqref{limits}.

As for the $\Gamma$-liminf inequality, we have to prove that 
\begin{equation}\label{liminfLq2}
\mathscr F'(u, A)\ge \mathscr F^\mu(u,A)\,.
\end{equation}
Let $(v_j)$ be a sequence converging to $u$ in $L^0(\Omega)$,
with $\mathscr F'(u,A)=\liminf_{j} \mathscr F_{K_j}(v_j,A)$. By the dominated convergence theorem, $v_j^t$ converges to $u^t$ in $L^p(\Omega)$ for every given $t>0$. Therefore, by assumption and by Remark~\ref{rmq:trunc}  we have
\[
\liminf_{j\to\infty} \mathscr F_{K_j}(v_j,A)\ge \liminf_{j\to\infty} \mathscr F_{K_j}(v_j^t,A)\ge \mathscr F^\mu(u^t,A).
\]
Taking the limit as $t\to\infty$ we get \eqref{liminfLq2}, thanks to Remark~\ref{convpqe} and Fatou's lemma.
\medskip

Conversely, assume $\Gamma$-convergence in $L^0(\Omega)$ and let us prove $\Gamma$-convergence in $L^q(\Omega)$.
To this aim we introduce the functionals $\mathscr F'_q(\cdot,A)$ and $\mathscr F''_q(\cdot,A)$, defined on $L^q(\Omega)$ as in \eqref{233}
with $L^0(\Omega)$ replaced by $L^q(\Omega)$.
The proof is similar to that of~\cite[Lemma 7.2]{DMI}.
The $\Gamma$-liminf inequality is immediate: for every $u\in L^q(\Omega)$ and for every
sequence $u_j\in L^q(\Omega)$ converging to $u$ in $L^q(\Omega)$, clearly $u_j$ converges to $u$
in $L^0(\Omega)$ too, hence by assumption
\(
\liminf_{j}\mathscr F_{K_j}(u_j,A)\ge\mathscr F^\mu(u,A)
\)
which implies that
\[
\mathscr F'_q(u,A)\ge \mathscr F^\mu(u,A)\,.
\]

As for the $\Gamma$-limsup inequality, we have to prove that for every $u\in L^q(\Omega)$ we have
\begin{equation}\label{limsupLq}
\mathscr F^\mu(u,A)\ge \mathscr F''_q(u,A)\,.
\end{equation}
Let $t>0$. By assumption,
there exists a sequence $(v_j)$, converging  to the truncation $u^t$  in $L^0(\Omega)$, such that $\mathscr F^\mu(u^t,A)\ge \Limsup_j \mathscr F_{K_j}(v_j,A)$; by Remark~\ref{rmq:trunc} we may also assume that
$v_j$ is bounded by $t$ in $L^\infty(\Omega)$. Hence $(v_j)$ converges to $u^t$  in $L^q(\Omega)$ and
consequently $\mathscr F^\mu(u^t,A)\ge \mathscr F''_q(u^t,A)$.
 Since $\mathscr  F^\mu(u,A)\ge \mathscr F^\mu(u^t,A)$ we have
$
\mathscr F^\mu(u,A)\ge \mathscr F''_q(u^t,A).
$
Since $u^t\to u$ in $L^q(\Omega)$ as $t\to\infty$, the lower semicontinuity of $\mathscr F''_q(\cdot,A)$  implies \eqref{limsupLq}.
\end{proof}

We now prove a compactness theorem ensuring $\Gamma$-convergence of the functionals \eqref{eq:1.2} to a $\Gamma$-limit with the integral representation given in \eqref{FmuA}

\begin{teo}\label{mainthm}
There exist a subsequence
of $(K_j)$, not relabelled, a measure $\mu\in \mathcal{M}_p(\Omega;M)$ and a rich family $\mathscr R\subset \mathscr A(\Omega)$ such that the sequence of functionals $(\mathscr F_{K_j}(\cdot,A))$ defined by \eqref{eq:1.2} $\Gamma$-converges in $L^0(\Omega)$ to the functional $\mathscr F^\mu(\cdot,A)$ defined by \eqref{FmuA} for all $A\in\mathscr R$. The measure $\mu$ possibly depends on $f$ and on the subsequence $(K_j)$ but not on $\mathscr R$. Moreover, $A\in \mathscr R$ whenever $C_{1,p}(\partial A\cap M)=0$, thus $\Omega\in \mathscr R$.
\end{teo}
\begin{proof}
By Remark~\ref{richrmq} and by \cite[Theorem 7.8, Remark 7.9 and Theorem 8.2]{Cor},
there exist a subsequence
of $(K_j)$, not relabelled, a measure $\mu\in \mathcal M_p(\Omega;M)$ and a rich family $\mathscr R\subset \mathscr A(\Omega)$ such that the restrictions to $L^p(\Omega)$ of the functionals $\mathscr F_{K_j}(\cdot,A)$ defined by \eqref{eq:1.2} $\Gamma$-converge in $L^p(\Omega)$ to the restriction to $L^p(\Omega)$ of the functional $\mathscr F^\mu(\cdot,A)$ for all $A\in \mathscr R$. Hence the conclusion follows thanks to Lemma~\ref{lm:p0}.
\end{proof}

\subsection{Convergence of minimisers}

This subsection is devoted to the convergence of the minimisers when the functionals $\mathscr F_{K_j}$ and $\mathscr F^\mu$ appearing in Theorem~\ref{mainthm} are perturbed with the functional $\mathscr G\colon L^0(\Omega)\times \mathscr{A}(\Omega)\to[0,\infty]$ defined by
\begin{equation*}
\mathscr G(u,A) := \begin{cases}
			\displaystyle\int_{A} g(x,u)\,dx\,, & \qquad\text{if $u\in L^{q}(A)$,}\\
			\infty\,,& \qquad \text{otherwise,}
\end{cases}
\end{equation*}
with $g$ satisfying \eqref{C}. 
A significant instance of a function satisfying \eqref{C} is $g(x,s)=|s-h(x)|^q$ for some $h\in L^q(\Omega)$. 
Notice that the compactness in $L^q(\Omega)$ of sequences of minimisers is not immediate even in this case,
except when $h$ is also bounded so that minimisers
satisfy apriori uniform $L^\infty$ estimates and their convergence to a minimiser is a mere consequence
of a $\Gamma$-convergence result (see Proposition~\ref{Gconvg} below).

\begin{teo}\label{mainthm2}
Let  $\mu\in\mathcal{M}_p(\Omega,M)$ and let $A\in \mathscr A(\Omega)$.
Assume that $(\mathscr F_{K_j}(\cdot, A))$ $\Gamma$-converge in $L^0(\Omega)$ to $\mathscr F^\mu(\cdot, A)$.
Then every sequence $(u_j)$ of minimisers of the problems
\begin{equation}\label{minj}
\min_{u\in L^{1,p}(A\meno K_j)} \Big\{ \int_{A\meno K_j} f(\nabla u)\,dx  + \int_A g(x,u)\,dx\Big\}
\end{equation}
has a subsequence which converges in $L^q(\Omega)$ to a minimiser of the problem
\begin{equation}\label{minj0}
\min_{u\in L^{1,p}(A\meno M)} \Big\{ \int_{A\meno M} f(\nabla u)\,dx + \int_{A\cap M} [u]^p\,d\mu + \int_A g(x,u)\,dx
\Big\}\,.
\end{equation}
Moreover, the minimum values of \eqref{minj} converge to the minimum value of \eqref{minj0}.
\end{teo}

In order to prove Theorem~\ref{mainthm2}, we need the following results about the $\Gamma$-convergence in $L^0(\Omega)$
of the functionals $\mathscr F_{K_j}+\mathscr G$ to $\mathscr F^\mu+\mathscr G$. Notice that the conclusion is not obvious because $\mathscr G$  is,  in general, not continuous in $L^0(\Omega)$.

\begin{prop}\label{Gconvg}
Let $\mu\in\mathcal{M}_p(\Omega,M)$ and let $A\in \mathscr A(\Omega)$. Assume that $(\mathscr F_{K_j}(\cdot,A))$ $\Gamma$-converge in $L^0(\Omega)$ to $\mathscr F^\mu(\cdot,A)$. Then, the sequence of functionals
$(\mathscr F_{K_j}(\cdot, A)+\mathscr G(\cdot, A))$ $\Gamma$-converge in $L^0(\Omega)$  to the functional
$\mathscr F^{\mu}(\cdot, A)+\mathscr G(\cdot, A)$.
\end{prop}
\begin{proof}
Let $u\in L^0(A)$.
By assumption and 
by the lower semicontinuity of the functional $\mathscr G$
with respect to the convergence in $L^0(\Omega)$, which follows from Fatou's Lemma, we have
\begin{equation}
\label{Glimsupinf}
\begin{split}
	\mathscr F^{\mu}(u,A)+\mathscr G(u,A)
	& \le  \inf\Big\{ \liminf_{j\to\infty} \big(\mathscr F_{K_j}(u_j,A)+\mathscr G(u_j,A)\big)\colon
	\text{$u_j\to u$ in $L^0(\Omega)$}  \Big\}\\
	& \le
	\inf\Big\{ \Limsup_{j\to\infty} \big(\mathscr F_{K_j}(u_j,A)+\mathscr G(u_j,A)\big)\colon	\text{$u_j\to u$ in $L^0(\Omega)$}  \Big\}\,.
\end{split}
\end{equation}
We claim that these inequalities are in fact equalities.
Indeed, if  $u\not\in L^q(A)$, by \eqref{C3} the left-hand side in \eqref{Glimsupinf} is infinite and the conclusion is obvious.
If instead $u\in L^q(A)$ then by Lemma~\ref{lm:p0} 
and by the continuity of the functional $\mathscr G$ with respect to the strong convergence
in $L^q(\Omega)$, we have that
\[
\inf\Big\{ \Limsup_{j\to\infty} \Big(\mathscr F_{K_j}(u_j,A)+\mathscr G(u_j,A)\Big)\colon	\text{$u_j\to u$ in $L^q(\Omega)$}  \Big\}
\le \mathscr F^{\mu}(u,A) + \mathscr G(u,A)\,.
\]
Since clearly the convergence in $L^q(\Omega)$ implies the convergence in $L^0(\Omega)$, we deduce the equalities in \eqref{Glimsupinf}.
This concludes the proof.
\end{proof}

In order to deduce the convergence of minimisers in $L^q(\Omega)$, we also need the following result.

\begin{prop}\label{prop:ptw}
Let $A\in \mathscr A(\Omega)$, let $u\in L^q(A)$, let $(u_j)$ be a sequence converging to $u$ in $L^0(\Omega)$ and assume that
\begin{equation}\label{convergence}
\lim_{j\to\infty} \int_{A}g(x,u_j)\,dx= \int_{A}g(x,u)\,dx\,.
\end{equation}
Then $(u_j)$ converges to $u$ strongly in $L^q(\Omega)$.
\end{prop}
\begin{proof}
By \eqref{C2}, the convergence of $(u_j)$ to $u$ in $L^0(\Omega)$ implies that $g(x,u_j)\to g(x,u)$ in $L^0(\Omega)$.
Thus, by the generalised dominated convergence theorem, using \eqref{convergence} and the lower bound in \eqref{C3} we deduce that
\begin{equation}\label{convnorm}
	\lim_{j\to\infty} \int_A |u_j|^q\,dx = \int_A |u|^q\,dx\,.
\end{equation}
By using again the generalised dominated convergence theorem, \eqref{convnorm} implies the strong convergence in $L^q(\Omega)$ of $(u_j)$ to $u$.
\end{proof}

\begin{proof}[Proof of Theorem~\ref{mainthm2}]
Let $u_j\in L^{1,p}(A \meno K_j)$ be a sequence of minimisers of the minimum problems \eqref{minj}.
By \cite[Proposition 7.1]{DMb}, $\sup_j [\mathscr F_{K_j}(u_j,A)+\mathscr G(u_j,A)]<\infty$, hence by \eqref{Ac} and \eqref{C3} there exists a function $u\in L^{1,p}(A\meno M)\cap L^q(A)$ and a subsequence, not relabelled, 
$( u_j)$ converging to $ u$ weakly in $W^{1,p}(A\meno M_\rho)$, for every $\rho>0$,
where 
\(
M_\rho = \left\{ x\in \Omega \colon {\rm dist}(x,\Sigma)\le\rho \right\}
\). Therefore, $(u_j)$ converges  in $L^0(\Omega)$  to a function $u\in L^q(A)$.
Moreover, as a general consequence of the $\Gamma$-convergence result of  Proposition~\ref{Gconvg},
see~\cite[Corollary 7.20]{DMb}, $u$ is a solution of the minimum problem \eqref{minj0} and
\begin{equation}\label{minTOmin2}
\lim_{j\to\infty} \Big\{\int_{A\setminus K_j} f(\nabla u_j)\,dx+ \int_A g(x,u_j)\,dx\Big\}
=\int_{A\setminus M}f(\nabla u)\,dx+ \int_{A\cap M} [u]^p\,d\mu+\int_A g(x,u)\,dx\,.
\end{equation}

Then, by the $\Gamma$-convergence assumption 
\begin{equation}
\label{minj3}
\int_{A\meno M} f(\nabla u)\,dx +  \int_{A\cap M} [u]^p\,d\mu
\le  \liminf_{j\to\infty} \int_{A\meno K_j} f(\nabla u_j)\,dx\,.
\end{equation}
The lower semicontinuity of $\int_A g(x,\cdot)dx$ with respect to the convergence in $L^0(\Omega)$ implies
\begin{equation}
\label{minj2}
\int_A g(x,u)\,dx\le \liminf_{j\to\infty} \int_A g(x,u_j)\,dx\,.
\end{equation}
Combining \eqref{minj3} with \eqref{minj2} we get
\[
\int_{A\meno M} f(\nabla u)\,dx + \int_{A\cap M} [u]^p\,d\mu+ \int_A g(x,u)\,dx 
\le \liminf_{j\to\infty}\Big\{ \int_{A\meno K_j} f(\nabla u_j)\,dx  + \int_A g(x,u_j)\,dx\Big\}\,,
\]
and by \eqref{minTOmin2} all the inequalities are in fact equalities. Hence by \eqref{minj2} we have
\[
	\int_A g(x,u) =\lim_{j\to\infty}\int_A g(x,u_j)\,dx\,.
\]
Thus, by Proposition~\ref{prop:ptw} we can deduce that $(u_j)$ converges to $u$ strongly in $L^q(\Omega)$, and this implies the desired conclusion.
\end{proof}

\section{Approximation of capacitary measures}\label{sec:density}

This section is devoted to prove that all capacitary measures concentrated on smooth hypersurfaces can be approximated by homogeneisation as described in Section~\ref{sec:3}. This density result means that the class  $\mathcal L_p(\Omega;M)$ of limit measures, i.e., limit measures that can appear in the conclusion of Theorem~\ref{mainthm} (see the following definition), coincides with the whole class $\mathcal M_p(\Omega;M)$ of capacitary measures concentrated on $M$.

\begin{dfn}\label{class:mu}
Let $\mu\in \mathcal{M}_p(\Omega; M)$ and let $A\in\mathscr A(\Omega)$. We say that $\mu\in\mathcal L_p(A; M)$ if 
there exists a sequence of compact sets $(K_j)$ satisfying \eqref{eq:1.1sec3} and such that
the sequence of functionals $(\mathscr F_{K_j}(\cdot,A) )$ defined by \eqref{eq:1.2} $\Gamma$-converge in $L^0(\Omega)$ to
the functional $\mathscr F^\mu(\cdot,A)$ defined as in \eqref{FmuA}.
In this case, we say that the sequence $(K_j)$ is associated with $\mu$.
\end{dfn}

\subsection{Stability results}
A sufficient condition for the stability of the class introduced in Definition~\ref{class:mu} is the $\Gamma$-convergence
of the corresponding functionals $\mathscr F^\mu$.

\begin{lm}
\label{Gamma-stability}
Let  $A\in\mathscr A(\Omega)$. 
Let $\mu^k \in\mathcal{L}_p(A; M)$ for every $k\in\mathbb N$, and let $\mu\in\mathcal M_p(A; M)$.
Assume that $(\mathscr F^{\mu^k}(\cdot, A))$ $\Gamma$-converge in $L^0(\Omega)$ to $\mathscr F^\mu(\cdot, A)$. Then $\mu\in\mathcal{L}_p(A;M)$.
\end{lm}
\begin{proof}
Given a functional $\mathscr F\colon L^q(\Omega)\times \mathscr{A}(\Omega)\to[0,\infty]$,
for every $A\in\mathscr A(\Omega)$ and for every $\tau>0$ we define a Moreau-Yosida-type approximation of index $\tau$ of $\mathscr F(\cdot,A)$ setting
\begin{equation*}
	\mathcal Y_\tau \mathscr F(u,A)  = \inf \big\{ \mathscr F(v,A) +\tau \|u-v\|^q_{L^q(A)}: v\in L^q(A)\big\}\,,
\end{equation*}
for all $u\in L^q(A)$.

Since $\mu^k\in\mathcal{L}_p(A;M)$, by Definition~\ref{class:mu}, there exists a sequence of compact sets $(K_j^k)$, with $\max\{{\rm dist}(x,M)\colon x\in K_j^k\}\to0^+$ as $j\to\infty$, such that the sequence of functionals $(\mathscr F_{K_j^k}(\cdot, A))$ defined by \eqref{eq:1.1sec3} $\Gamma$-converge in $L^0(\Omega)$  to the functional $\mathscr F^{\mu^k}(\cdot, A)$. Possibly passing to a subsequence we may assume $\max\{{\rm dist}(x,M)\colon x\in K_j^k\}\le 1/j$. By Theorem~\ref{mainthm2} 
\begin{equation}\label{moryo}
	\lim_{j\to\infty} \mathcal{Y}_\tau \mathscr F_{K_j^k}(u,A) = \mathcal{Y}_\tau \mathscr F^{\mu^k}(u,A)\quad
	\text{for every $u\in L^q(A)$, $\tau>0$, and $k\in\mathbb N$.}
\end{equation}
As in Lemma~\ref{lm:p0}, by the $\Gamma$-convergence assumption we can prove that the restrictions to $L^q(\Omega)$ of $(\mathscr F^{\mu^k}(\cdot, A))$ $\Gamma$-converge in $L^q(\Omega)$ to the restriction to $L^q(\Omega)$ of $\mathscr F^{\mu}(\cdot, A)$. Then arguing as in Theorem~\ref{mainthm2} we obtain
\begin{equation}\label{moryo2}
	\lim_{k\to\infty} \mathcal{Y}_\tau \mathscr F^{\mu^k}(u,A) = \mathcal{Y}_\tau \mathscr F^{\mu}(u,A) \quad \text{for every $u\in L^q(A)$, $\tau>0$, and $k\in\mathbb N$.}
\end{equation}
Let $(u_i)$ be a countable dense sequence in $L^q(A)$. By \eqref{moryo} for every $k$ there exists $j_k\ge k$ such that 
\[
\left|\mathcal{Y}_\tau \mathscr F_{K_{j_k}^k}(u_i,A) - \mathcal{Y}_\tau \mathscr F^{\mu^k}(u_i,A)\right|\le 1/k\,,
\]
every $i=1,\dots, k$ and every $\tau=1,\dots, k$. By \eqref{moryo2} this implies 
\[
	\lim_{k\to\infty} \mathcal{Y}_\tau \mathscr F_{K_{j_k}^k}(u_i,A) = \mathcal{Y}_\tau \mathscr F^\mu(u_i,A)\,,\qquad
	\text{for every $i,\tau\in\mathbb N$.}
\]
From \cite[Theorem 9.16]{DMb}, using Lemma~\ref{lm:p0} again, we deduce that the sequence of functionals $(\mathscr F_{K_{j_k}^k}(\cdot, A))$
$\Gamma$-converge in $L^0(\Omega)$  to the functional $\mathscr F^{\mu}(\cdot, A)$. Hence to conclude that $\mu\in\mathcal L_p(A;M)$, it suffices to
observe that  $\max\{{\rm dist}(x,M)\colon x\in K_{j_k}^k\}\le 1/j_k\le1/k$  for $j$ large enough.
\end{proof}

A simple condition for the $\Gamma$-convergence of the functionals $\mathscr F^{\mu}$ is provided by the following lemma.

\begin{lm}\label{lm:supernuovo}
Let  $A\in\mathscr A(\Omega)$. 
Let $\mu^k\in \mathcal{M}_p(A;M)$ for every $k\in \mathbb N$, and let $\mu\in \mathcal M_p(A;M)$. Assume that $(\mu^k)$ is monotone and that $\mu$ is the pointwise limit of the measures $\mu^k$, defined as $\mu(B):=\lim_{k\to\infty} \mu^k(B)$ for every $B\in\mathscr B(A)$.  
Then $(\mathscr F^{\mu^k}(\cdot, A))$ $\Gamma$-converges in $L^0(\Omega)$ to $\mathscr F^\mu(\cdot, A)$.
\end{lm}
\begin{proof}
In view of~\cite[Proposition 5.4]{DMb} (see also~\cite[Remark 5.5]{DMb}), 
if
$(\mu^k(B))$ is a monotone increasing sequence for every $B\in\mathscr B(A)$ then
the $\Gamma$-limit of the functionals $\mathscr F^{\mu^k}$ coincides with the pointwise monotone limit, which by assumption is given by $\mathscr F^\mu$. 
\end{proof}

An immediate consequence of Lemma~\ref{Gamma-stability} and Lemma~\ref{lm:supernuovo} is that a measure belongs to $\mathcal{L}_p(A;M)$
whenever it can be written as the pointwise monotone limit of measures of class $\mathcal{L}_p(A;M)$.

The stability of the class introduced in Definition~\ref{class:mu} holds also under convergence in dual Sobolev spaces. To see this we need the following result about the lower semicontinuity of the integral of the jump of $u$ when both the measure and the function vary.

\begin{lm}\label{lm:76}
Let $\Omega_0$ and $\Sigma_0$ be defined by \eqref{cilynder} and \eqref{Ngraph}, respectively. 
Let $(\mu^k)$ be a sequence of non-negative Radon measures in $W^{-1,p'}(\Omega_0)$ converging to $\mu$ strongly in $W^{-1,p'}(\Omega_0)$ and let $(u_k)$ be a sequence of functions converging to $u$ weakly in $W^{1,p}(\Omega_0\meno \Sigma_0)$.
Then
\begin{equation}
\label{eq:76}
\int_{\Sigma_0} [u]^p\,d\mu\le\liminf_{k\to\infty} \int_{\Sigma_0}[u_k]^p\,d\mu^k\,.
\end{equation}
\end{lm}
\begin{proof}
It suffices to prove that
\begin{equation}\label{lm7:claim}
\int_{\Sigma_0} |\tilde z|^p\, d\mu \le \liminf_{k\to\infty}
 \int_{\Sigma_0}
	|\tilde z_k|^p\,d\mu^k
\end{equation}
for all sequences $(z_k)$ weakly converging to $z$ in $W^{1,p}(\Omega_0)$.
Let $(\varphi_i)$ be a sequence in $C^\infty_c(\Omega_0)$ such that $0\le\varphi_i\le \varphi_{i+1}\le 1$ for all $i$ and $\varphi_i(x)\to 1$ for all $x$.
Let $t>0$ and let $z_k^t$ and $z^t$ be the  truncations of $z_k$ and $z$.
Since $\mu^k$ converges to $\mu$ strongly in $W^{-1,p'}(\Omega_0)$,
$|z_k^t|^p\,\varphi_i$ converges to $|z^t|^p\,\varphi_i$ weakly in 
$W^{1,p}_0(\Omega_0)$, and
\[
	\langle \mu^k\!,|z_k^t|^p\varphi_i\rangle = \! \int_{\Sigma_0}\!\!\!\!
	|\tilde z_k^t|^p\,\varphi_i\,d\mu^k,\quad \langle \mu,|z^t|^p\varphi_i\rangle = \! \int_{\Sigma_0}\!\!\!\!
	|\tilde z^t|^p\,\varphi_i\,d\mu\,,
\]
we obtain
\begin{equation}\label{lm7:claim22}
\int_{\Sigma_0} |\tilde z^t|^p\,\varphi_i\, d\mu \le \liminf_{k\to\infty}
 \int_{\Sigma_0}
	|\tilde z_k^t|^p\,\varphi_i\,d\mu^k\le \liminf_{k\to\infty}\int_{\Sigma_0}
	|\tilde z_k|^p\,d\mu^k\,.
\end{equation}
Letting $i\to\infty$ and $t\to\infty$, by the monotone convergence theorem we obtain \eqref{lm7:claim}  from \eqref{lm7:claim22}.

We now prove \eqref	{eq:76}. Let $\Omega_0^\pm$ the two connected components of $\Omega_0\setminus \Sigma_0$, as in \eqref{2Omega}. 
By applying a linear extension operator we find 
$v_k,v,w_k,w\in W^{1,p}(\Omega_0)$ such that
\begin{equation*}
	u_k^+ = \tilde v_k\,,  \quad u^+= \tilde v\,, \quad u_k^-=\tilde w_k\,, \quad u^- = \tilde w\quad \text{$p$-q.e. on $\Sigma_0$.}
\end{equation*} 
Let $z_k=\tilde v_k-\tilde w_k$ and $z=\tilde v-\tilde w$ so that 
\[
[u_k]=|z_k|\quad  \text{and}\quad [u]=|z| \quad  \text{$p$-q.e. on $\Sigma_0$.}
\]
By linearity $(z_k)$ converges to $z$ weakly in $W^{1,p}_0(\Omega_0)$, hence \eqref{eq:76} follows from \eqref{lm7:claim}.
\end{proof}

We now prove the closure of the class $\mathcal{L}_p(\Omega_0;M)$ under strong convergence in dual Sobolev space.

\begin{lm}\label{lm:77}
Let $\Omega_0$ be defined by \eqref{cilynder} and let $\Sigma_0=\Omega_0\cap \Sigma$ be representable as in \eqref{Ngraph}.
Let $\mu^k\in\mathcal{L}_p(\Omega_0;M)$ for every $k\in\mathbb N$ and let $\mu\in\mathcal M_p(\Omega_0;M)$.
Assume that $(\mu^k)$ is a sequence of non-negative Radon measures in $W^{-1,p'}(\Omega_0)$ converging to $\mu$ strongly in $W^{-1,p'}(\Omega_0)$. 
Then $\mu\in\mathcal{L}_p(\Omega_0;M)$.
\end{lm}
\begin{proof} 
By Lemma~\ref{Gamma-stability}, it suffices to prove the $\Gamma$-convergence in $L^0(\Omega)$ of
the sequence $(\mathscr F^{\mu^k}(\cdot, \Omega_0))$ to $\mathscr F^\mu(\cdot, \Omega_0)$.
Let $\widehat{\mathscr F}'$ and $\widehat{\mathscr F}''$ be the $\Gamma$-liminf and $\Gamma$-limsup of $\mathscr F^\mu$ in $L^0(\Omega)$ of this sequence. Given $u\in L^0(\Omega)$ and $t>0$ we consider the truncation $u^t$.
By extension theorems in Sobolev spaces we may find functions
$v,w\in W^{1,p}(\Omega_0)\cap L^\infty(\Omega_0)$ satisfying
\begin{equation*}
	 (u^t)^+= \tilde v\quad \text{and} \quad (u^t)^- = \tilde w\quad \text{$p$-q.e. on $\Sigma_0$.}
\end{equation*}
This implies that 
\[
\int_{\Sigma_0} [u^t]^p\,d\mu^k=\langle \mu^k,|v-w|^p\rangle\to\langle \mu,|v-w|^p\rangle=\int_{\Sigma_0} [u^t]^p\,d\mu
\]
which gives $$\widehat{\mathscr{F}}''(u^t,\Omega_0)\le \lim_{k\to\infty}\mathscr F^{\mu^k}(u^t,\Omega_0)=\mathscr F^{\mu}(u^t,\Omega_0)\le \mathscr F^{\mu}(u,\Omega_0).$$
Taking the limit as $t\to\infty$ we obtain $\widehat{\mathscr{F}}''(u,\Omega_0)\le \mathscr F^\mu(u,\Omega_0)$. On the other hand, we also have $ \mathscr F^\mu(u,\Omega_0)\le \widehat{\mathscr  F}'(u,\Omega_0)$. This is an immediate consequence of Lemma~\ref{lm:76} and of the lower semicontinuity of the integral $\int_{\Omega_0}f(\nabla u)\, dx$. Combining the two inequalities we get the desired $\Gamma$-convergence and we conclude.
\end{proof}

\subsection{Density results}

We begin by proving that the collection of all non-negative finite Radon measures on $M$ belonging to $W^{-1,p'}(\mathbb R^n)$ is dense in the class  of all capacitary measures concentrated on $M$.

\begin{lm}\label{lemmanuovo}
Let $\mu\in\mathcal M_p(\Omega, M)$ and let  $A\in\mathscr A(\Omega)$. Then there exists a sequence $(\mu^k)$ of non-negative finite
Radon measures of class  $W^{-1,p'}(\Omega)$ such that $(\mathscr F^{\mu^k}(\cdot, A))$ $\Gamma$-converges in $L^0(\Omega)$ to  $\mathscr F^\mu(\cdot, A)$.
\end{lm}
\begin{proof}
By Lemma~\ref{lm:equivmis} there exist a Borel function $\psi\colon M\to [0,\infty]$ and  a non-negative
Radon measure $\sigma$ of class $W^{-1,p'}(\mathbb R^n)$ such that 
\begin{equation}
\label{equivms2}
	\int_{A\cap M} [u]^p\,d\mu = \int_{A\cap M} [u]^p\psi\,d\sigma\,,
\end{equation}
for all $u\in L^{1,p}(A\setminus M)$. Since the measure $\psi\sigma$ is the pointwise limit of the measures
$\mu^k = (\psi\wedge k)\sigma$ the $\Gamma$-limit of $\mathscr F^{\mu^k}(\cdot,A)$ is $\mathscr F^\mu(\cdot,A)$, thanks to Lemma~\ref{lm:supernuovo} and \eqref{equivms2}. Clearly every $\mu^k$ is a finite non-negative Radon measure of class $W^{-1,p'}(\mathbb R^n)$.
\end{proof}

In view of Lemmas~\ref{Gamma-stability} and \ref{lemmanuovo}, to prove that $\mathcal L_p(A;M)=\mathcal M_p(A;M)$ it is enough to show that every
non-negative Radon measure belonging to $W^{-1,p'}(\mathbb R^n)$ also belong to $\mathcal L_p(A, M)$.
We first prove this conclusion in the special case of hyperplanes.

\begin{prop}\label{prop:density} 
Assume that $M$ and $\Sigma$ are contained in a hyperplane. Let $\Omega_0$ be as in \eqref{cilynder} so that $\Pi_0=\Omega_0\cap \Sigma$.
Then $\mathcal M_p(\Omega_0,M)=\mathcal L_p(\Omega_0;M)$.
\end{prop}
\begin{proof}
We only have to prove the inclusion $\mathcal M_p(\Omega_0;M)\subset\mathcal L_p(\Omega_0;M)$, the other being trivial by Definition~\ref{class:mu}. 
It is known that the measures $\beta\mathscr H^{n-1}\mres (A\cap M)$ belong to the class
$\mathcal{L}_p(\Omega_0; M)$ for all $\beta>0$ and for all $A\in\mathscr A(\Omega_0)$, see, e.g.,  \cite[Theorem 3.3]{ansini} (see also Remark~3.3(a) therein), for which the hypothesis $p\le n$ is crucial.
Since by Definition \ref{class:mu}
the class $\mathcal{L}_p(\Omega_0;M)$ is closed under sum and multiplication by positive scalars,
we deduce that
\[
\sum_{i=1}^h\beta_i 1_{A_i}\mathscr H^{n-1}\mres M\in\mathcal{L}_p(\Omega_0;M)
\] 
for all pairwise disjoint relative open subsets $A_1,\ldots,A_h$ of $M$, for all $\beta_1,\ldots,\beta_h>0$,
and for all $h\in\mathbb N$.
Since by standard density result simple functions are dense in $L^{p'}(\Omega_0)$,  by Lemma~\ref{lm:77} we can deduce that every measure of the form $\theta\mathscr H^{n-1}\mres M$, 
with $\theta\in L^{p'}(M;\mathscr H^{n-1})$ and $\theta\ge 0$,
belongs to $\mathcal{L}_p(\Omega_0;M)$. Indeed, if $(\theta_h)$ is a sequence converging to $\theta$ in $L^{p'}(M,\mathscr H^{n-1})$, then the sequence of non-negative Radon measures $\theta_h\mathscr H^{n-1}\mres  M$
converges to $\theta\mathscr H^{n-1}\mres  M$ in $W^{-1,p'}(\Omega_0)$, by H\"older inequality and Sobolev trace theorem.

Next we prove that the linear space $\mathcal{X}:=\{ \theta\mathscr H^{n-1}\mres  M\colon \theta\in C^\infty_0(\Omega_0) \}$
is dense, with respect to the strong 
topology of $W^{-1,p'}(\Omega_0)$, in the closed linear space $\mathcal{Y}:=\{F\in W^{-1,p'}(\mathbb R^n)\colon {\rm supp}(F)\subset  M\}$. 
If not, by Hahn-Banach theorem and by the reflexivity
of $W^{1,p}(\mathbb R^n)$ there exists $F_0\in \mathcal{Y}\setminus \mathcal X$ and $u_0\in W_0^{1,p}(\Omega_0)$
such that $\langle F_0,u_0\rangle\neq0$ and  $\langle F, u_0 \rangle=0$  for every $F\in\mathcal X$. The last condition implies  
\begin{equation*}
	\int_{ M} u_0\,\theta\,d\mathscr H^{n-1}=0\,,\qquad \text{for all $\theta\in C^\infty_0(\mathbb R^n)$.}
\end{equation*}
In particular, $u_0=0$ $\mathscr H^{n-1}$-a.e.\ on $ M$. Then, $u$ would be the limit
in $W^{1,p}(\Omega_0)$  of a sequence 
$(u_\varepsilon)\subset C^\infty_0(\Omega_0)$ such that ${\rm supp}(u_\epsilon)\cap M=\emptyset$. Since $F_0\in \mathcal{Y}$, then ${\rm supp}(F_0)\subset M$, hence $\langle F_0, u_\epsilon\rangle=0$ for every $\epsilon$. This implies that $\langle F_0, u_0\rangle=0$, and this contradiction concludes the proof of the density of $\mathcal X$ in $\mathcal Y$.

Using Lemma~\ref{lm:77} again, from the density of $\mathcal{X}$ in $\mathcal{Y}$
we deduce, in particular, that
all non-negative Radon measures of class $W^{-1,p'}(\Omega_0)$ with support contained in $ M$ belong to
$\mathcal{L}_p(\Omega_0; M)$.  Then, the desired conclusion follows by Lemma~\ref{Gamma-stability} and Lemma~\ref{lemmanuovo}.
\end{proof}

The following lemma is another ingredient in the proof of the main result of this section: it entails
that every capacitary measure is sufficiently close to a measure of the class introduced in Definition~\ref{class:mu}.

\begin{lm}\label{lm:diffeo2}
Let $\eta>0$, let $\Omega_0$ be a cylinder of the form \eqref{cilynder}, let $\Sigma_0=\Omega_0\cap \Sigma$ be representable as in \eqref{Ngraph}, and let $\Psi\colon \Omega_0\to \Omega_0$ be a  $C^1$ diffeomorphism with $\|\Psi-{\rm Id}\|_{C^1(\Omega_0)}<\eta$ and $\Psi(\Sigma_0) = \Pi_0$, where $\rm Id$ is the identity map.
Let $\mu\in \mathcal{M}_p(\Omega_0;M)$. Then there exists
$\sigma\in \mathcal{L}_p(\Omega_0;M)$ with
\begin{equation}\label{eq:diffeo2}
\left|\int_{\Sigma_0}[u]^p\,d(\mu-\sigma)\right|\le C \eta\int_{\Omega_0\setminus \Sigma_0}f(\nabla u)\,dx, \quad \text{for all $u\in L^{1,p} (\Omega_0\setminus M)$,}
\end{equation}
where $C>0$ is a constant only depending on $\lambda, \Lambda, p$.
In addition, if $E,E'$ are Borel subsets of $\Sigma_0$ with $\mu(\Sigma_0\setminus E)=0$ and
\begin{equation}\label{metric:separation}
\inf_{E\times E'}|x-y|>\delta>0,
\end{equation}
then $\sigma(E')=0$.
\end{lm}
\begin{proof}
We first observe that there exists a constant $C>0$ only depending on $\lambda,\Lambda,p$ such that 
\begin{equation}\label{smallness-f}
\left\vert\int_{\Omega_0\setminus K}\Big[f(\nabla u)-f\big(\nabla u(D\Psi^{-1}\circ \Psi)\big)|\det D\Psi(x)| \Big]\,dx\right\vert\le
 C\,\eta\, \int_{\Omega_0\setminus K} f(\nabla u)\,dx\,,
\end{equation}
for every closed set $K\subset \Sigma_0$  and for every $u\in W^{1,p}(\Omega_0\setminus K)$ (with the convention that $\nabla u$ is a row vector). Indeed, by \eqref{Ab} and Euler's identity we have 
\[
	\nabla_\xi f(\xi)\cdot v = \lim_{t\to0^+} \frac{f(\xi+t(v-\xi))-f(\xi)}{t} + \nabla_\xi f(\xi)\cdot \xi \le f(v)+(p-1)f(\xi)\,,
\]
for a.e. $\xi\in\mathbb R^n$, so that a standard homogeneity argument gives
\begin{equation*}
|  \nabla_\xi f(\xi)\cdot v| \le pf(\xi)^\frac{p-1}{p} f(v)^\frac{1}{p}\,,\quad \text{for a.e. $\xi\in\mathbb R^n$.}
\end{equation*}
Since the restriction of $f$ to a line is a convex function then it is
differentiable a.e.\ w.r.t.\ the one-dimensional Lebesgue measure. Hence,
plugging in $\xi=(1-t)\nabla u(D\Psi^{-1}\circ \Psi)+t\nabla u$ with $0<t<1$ and $v= \nabla u(I-D\Psi^{-1}\circ \Psi)$
in the last inequality and integrating, by the assumption that $\| \Psi - I \|_{C^1(\Omega_0)}<\eta$, we obtain
\begin{equation*}
\int_{\Omega_0\setminus K}\Big|f(\nabla u)- f(\nabla u(D\Psi^{-1}\circ \Psi)) \Big|\,dx\le
 C \eta \int_{\Omega_0\setminus K} f(\nabla u)\,dx\,,
\end{equation*}
(where we also used the convexity of $f$ again) with the constant being independent of $ f$, $H$ and $\eta$. Using the assumption again, \eqref{smallness-f} follows.

Now, let $u\in L^{1,p} (\Omega_0\setminus M)$ and let $\pi\colon \mathscr B(\Omega_0)\to [0,\infty]$ be the measure of $\mathcal{M}_p(\Omega_0;M)$ defined by
\begin{equation}
\label{pisharp}
\pi(B) := \mu(\Psi^{-1}(B)\cap \Omega_0)\,,\qquad \text{for all $B\in \mathscr B(\Omega_0)$.}
\end{equation}
Since $\psi(M)\cap \Omega_0$ is contained in $\Pi_0$, by Proposition~\ref{prop:density}, $\pi\in \mathcal{L}_p(\Omega_0;\Psi(M))$. Therefore, there exist a sequence of compact sets $(H_j)$ with $\max\{{\rm dist}(x,\Psi(M))\colon x\in H_j\}\to0^+$ as $j\to\infty$ and a sequence $(v_j)$ of $L^{1,p}(\Omega_0\setminus \Psi(M))$ converging to the function
 $v=u\circ\Psi^{-1}\in W^{1,p}(\Omega_0\setminus  \Psi(M))$ in $L^0(\Omega)$, such that
\begin{subequations}\label{eq:density}
\begin{equation}\label{eq1:density}
	\lim_{j\to\infty} \int_{\Omega_0\setminus H_j}  f(\nabla v_j)\,dx
	=\int_{\Omega_0\setminus \Pi_0}  f(\nabla v)\,dx +  \int_{ \Pi_0} [v]^p\,d\pi\,.
\end{equation}
Moreover, the sequence of compact sets $(K_j)$ and of functions $(u_j)$ of $W^{1,p}(\Omega_0\setminus M)$, defined by composition, respectively, as $K_j=\Psi^{-1}(H_j)$ and $u_j=v_j\circ \Psi$, are such that  $\max\{{\rm dist}(x,M)\colon x\in K_j\}\to0^+$ as $j\to\infty$ and $u_j$ converges to $u$ in $L^0(\Omega)$. Hence, by Theorem~\ref{mainthm}, up to subsequences (not relabelled) there exists $\sigma\in \mathcal{L}_p(\Omega_0;M)$ such that
\begin{equation}
\label{eq2:density}
\liminf_{j\to\infty} \int_{\Omega_0\setminus K_j} f(\nabla u_j)\,dx\ge \int_{\Omega_0\setminus \Sigma_0} f(\nabla u)\,dx+\int_{\Sigma_0} [u]^p\,d\sigma\,.
\end{equation}
\end{subequations}
Therefore, by combining \eqref{eq:density} with \eqref{pisharp} and by using $\Psi$ to change variables in the volume integrals, recalling \eqref{smallness-f}, we deduce the estimate
\begin{equation*}
\int_{\Sigma_0}[u]^p\,d(\sigma-\mu)=\int_{ \Sigma_0} [u]^p\,d\sigma - \int_{ \Pi_0} [v]^p\,d\pi\le C\eta\,
\int_{\Omega_0\setminus\Sigma_0}f(\nabla u)\,dx\,,
\end{equation*}
The estimate \eqref{eq:diffeo2} then follows by estimating the integral $\int_{\Sigma_0}[u]^p\,d(\mu-\sigma)$ in a similar way, by choosing a recovering sequence for $\mathscr F^\sigma(u,\Omega)$ and by using $\Psi$ to change variables with the estimate \eqref{smallness-f}.

As for the final part, from the assumption on $\mu$ and definition \eqref{pisharp} we deduce that $\pi$ is concentrated on $\Psi(E)\cap \Pi_0$. Therefore,
by Definition~\ref{class:mu} we may assume the closed sets $H_j$ appearing in \eqref{eq1:density} 
to be contained in $\Psi(E)$. Since $\Psi$ is a diffeomorphism, every set $\Psi^{-1}(H_j)$ is contained in $E$. Then the measure
$\sigma$ appearing in \eqref{eq2:density} must be concentrated 
in $\{x\in\Sigma_0\colon {\rm dist}(x,E)<\delta\}$, which by \eqref{metric:separation} implies $\sigma(E')=0$.
\end{proof}

We are now in position to prove the main result of this section.

\begin{teo}
The following equality holds: $\mathcal{L}_p(\Omega;M) = \mathcal{M}_p(\Omega;M)$.
\end{teo}

\begin{proof}
In  view of Definition~\ref{class:mu}, we fix a measure $\mu\in \mathcal{M}_p(\Omega;M)$ and we prove that $\mu\in \mathcal{L}_p(\Omega;M)$. 
We can also assume $\mu$ to be a finite Radon measure of class $W^{-1,p'}(\mathbb R^n)$, thanks to Lemma~\ref{Gamma-stability} and Lemma~\ref{lemmanuovo}.
Let $k>0$. We consider a finite family $\{E_i\}$ of pairwise disjoint Borel sets of $M$ such that $M=E_0\cup E_1\cup \ldots\cup E_m$ for some $m\in\mathbb N$,  $\mu(E_0)<2^{-k}$, every $E_i$ (with $i\ge1$) is contained in a cylinder $\Omega_i$ of the form \eqref{cilynder} and $\Sigma_i=\Sigma\cap \Omega_i$ admits a graph representation as in \eqref{Ngraph}, for suitable radii $r_i>0$, centres $x_i\in \Sigma$ and axis $\nu_i\in\mathbb R^n$ with $|\nu_i|=1$. Moreover, for $i\ge1$, the sets $E_i$ can be chosen to satisfy
\begin{equation*}
\min_{1\le i<j\le m}\inf_{(x,y)\in E_i\times E_j}|x-y|>0\,.
\end{equation*}
Let $\mu_i:=\mu\mres E_i\in W^{-1,p'}(\Omega_i)$ for every $i=1,\ldots,m$. By Lemma~\ref{lm:diffeo2}, there exists $\sigma_i^k\in \mathcal{L}_p(\Omega_i;M)$, concentrated on $E_i$, such that
\begin{equation}\label{eq:diffeoeta}
\left|\int_{\Sigma_i}[u]^p\,d(\mu_i-\sigma_i^k)\right|\le C 2^{-k}\int_{\Omega_i\setminus \Sigma_i}f(\nabla u)\,dx\,,
\qquad \text{for all $u\in W^{1,p}(\Omega_i\setminus M)$.}
\end{equation}
The measure $\sigma^k = \sum_{i=1}^m\sigma_i^k$ is a non-negative Borel measure
that belongs to $\mathcal{L}_p(\Omega;M)$, since the summands are concentrated on pairwise
disjoint Borel sets by construction. Therefore, since $\mu(E_0)<2^{-k}$,  from  \eqref{eq:diffeoeta} we deduce that
\begin{equation}\label{4.19}
\left|\int_{ M}[u]^p\,d(\mu-\sigma^k)\right|\le C\Big[\int_{\Omega\setminus M} f(\nabla u)\,dx+2t^p\Big]2^{-k}\,,
\end{equation}
for all $u\in W^{1,p}(\Omega\setminus M) $ with $\|u\|_{L^\infty(\Omega)}\le t$ with $t>0$, where we also used
$[u]=0$ $p$-q.e.\ on $\Sigma\setminus M$.

By Lemma~\ref{Gamma-stability} again, it suffices to prove the $\Gamma$-convergence in $L^0(\Omega)$ of
the sequence $(\mathscr F^{\sigma^k}(\cdot, \Omega))$ to $\mathscr F^\mu(\cdot, \Omega)$. Let $\widehat{\mathscr F}'$ and $\widehat{\mathscr F}''$ be the $\Gamma$-liminf and $\Gamma$-limsup of $\mathscr F^\mu$ in $L^0(\Omega)$ of this sequence. Given $u\in L^0(\Omega)$ and $t>0$ we consider the truncation $u^t$. The $\Gamma$-limsup is immediate. By \eqref{Ac} we can assume $u\in L^{1,p}(\Omega\setminus M)$ otherwise the inequality is trivial. Therefore, by \eqref{4.19} we have
\[
\widehat{\mathscr F}''(u^t,\Omega)\le \mathscr F^\mu(u^t,\Omega)\le \mathscr F^\mu(u,\Omega)\,.
\]
Taking the limit as $t\to\infty$ we obtain $\widehat{\mathscr{F}}''(u,\Omega)\le\mathscr F^\mu(u,\Omega)$. On the other hand, for the $\Gamma$-liminf inequality we have to prove that
\begin{equation}\label{liminfsez4}
\widehat{\mathscr{F}}'(u,\Omega)\ge\mathscr F^\mu(u,\Omega)\,.
\end{equation}
Let $(u_k)$ be a sequence converging to $u$ in $L^0(\Omega)$ with $\widehat{\mathscr F}'(u,\Omega)=\liminf_{k} \mathscr F^{\sigma^k}(u_k,\Omega)$. Moreover, we can assume this sequence to be equibounded in $L^{1,p}(\Omega\setminus M)$ so that $(u_k^t)$ converges, up to subsequences (not relabelled), to $u$ weakly in $W^{1,p}(\Omega\setminus M)$.  Therefore, as $\mathscr F^{\sigma^k}(\cdot,\Omega)$ decreases under truncation, the estimate \eqref{4.19},  the lower semicontinuity of the integral $\int_{\Omega_0}f(\nabla u)\, dx$ and Lemma~\ref{lm:76} applied to each $\Omega_i$ yields
\[
\liminf_{k\to\infty} \mathscr F^{\sigma^k}(u_k,\Omega)\ge \liminf_{k\to\infty} \mathscr F^{\sigma^k}(u_k^t,\Omega)\ge \mathscr F^\mu(u^t,\Omega).
\]
Taking the limit as $t\to\infty$ we get \eqref{liminfsez4}, thanks to Remark~\ref{convpqe} and Fatou's lemma.
\end{proof}

\section{Construction of capacitary measures}\label{reconstruction}

In this section we provide a general procedure to construct the capacitary measure  $\mu\in \mathcal M_p(\Omega;M)$ appearing in the conclusion of Theorem~\ref{mainthm}. This will be done by solving some auxiliary localised minimum problems 
involving the function $f$ and the sequence of compact sets $K_j$.

\subsection{Formulas for capacitary measures}

Let $\Omega_0$ be a cylinder of the form \eqref{cilynder} for some $x_0\in \mathbb{R}^n$, $\nu_0\in \mathbb{R}^n$ with $|\nu_0|=1$, and $r_0>0$. Let $\Sigma_0=\Sigma\cap \Omega_0$ be a graph representation as in \eqref{Ngraph}. For every $\rho\in\left(0,\tfrac{r_0}{2}\right)$ we introduce $\Sigma_\rho^\pm := \Sigma_0\pm\rho\nu_0$ and  
$$S_\rho := \{\bar x +r\nu_0\colon \bar x\in \Pi_0\,,\ |r- \phi(\bar x)|<\rho\}\,.$$
Moreover, for every $A\in\mathscr A(\Omega_0)$ and every $\mu\in\mathcal{M}_p(\Omega_0;M)$ we set
\begin{equation}\label{mrhoa}
m_\rho(A,\mu):=\displaystyle\min_{v\in\mathcal{V}_{\rho}(A)}
\int_{A\cap S_{\rho}\setminus \Sigma_0}  f(\nabla v)\,dx+ \int_{A\cap \Sigma_0} [v]^p\,d\mu\,,
\end{equation}
where $\mathcal{V}_{\rho}(A)$ denotes the set of all functions $v\in L^{1,p}(A\cap S_\rho\meno \Sigma_0)$
with $v= 1$ $\mathscr{H}^{n-1}$-a.e.\ on $A\cap \Sigma_\rho^+$
and $v=0$ $\mathscr{H}^{n-1}$-a.e. on $A\cap \Sigma_\rho^-$. The existence of a minimiser in \eqref{mrhoa} can be proved by the direct methods of the Calculus of Variations. By truncation every minimiser of \eqref{mrhoa} is bounded by 0 and 1 and therefore it belongs to $W^{1,p}(A\cap S_\rho\meno M)$.
Clearly \eqref{mrhoa} defines a non-decreasing set function.
We claim that it is inner regular, i.e.,
\begin{equation}\label{inner}
m_\rho(A,\mu)=\sup_{\substack{A'\in \mathscr{A}(\Omega) \\ A' \subset\subset A }} m_\rho(A',\mu)\,.
\end{equation}
The inequality $\ge$ is trivial by monotonicity of $m_\rho$ with respect to set inclusion. To prove the reverse inequality $\le$, we may assume that the $\sup$ is finite. We consider an increasing sequence of open sets $A_i\subset\subset A$ along which the sup in \eqref{inner} is achieved as a limit and $u_i\in \mathcal{V}_\rho(A_i)$ is chosen to attain the minimum in \eqref{mrhoa} with $A$ replaced by $A_i$.
Let $A''\in \mathscr{A}(\Omega)$ be fixed with $A''\subset \subset A$. Then, since clearly $A''\subset A_i$ for all large indices $i$, we have
\[
\sup_{\substack{A'\in \mathscr{A}(\Omega) \\ A' \subset\subset A }} m_\rho(A',\mu)\ge \liminf_{i\to\infty} \int_{A''\cap \Sigma_0} f(\nabla u_i)
	+\int_{A''\cap \Sigma_0}[u_i]^p\,d\mu\,.
\]
By construction $(u_i)$ converges (up to subsequences) to a function $u$ in $L^0(\Omega)$ (the functions $u_i$ are
extended by zero to the whole $A$) 
and $(\nabla u_i)$ converges to $\nabla u$ weakly in $L^{p}(A\cap S_\rho\setminus \Sigma_0)$.  By Fatou's lemma  and by Lemma~\ref{lm:76} 
we deduce that
\[
\sup_{\substack{A'\in \mathscr{A}(\Omega) \\ A' \subset\subset A }} m_\rho(A',\mu) \ge  \int_{A''\cap S_\rho\setminus \Sigma_0} f(\nabla u)
	+\int_{A''\cap \Sigma_0}[u]^p\,d\mu\,,
\]
which, by the arbitrariness of $A''\subset\subset A$ and the monotone convergence theorem, implies
\[
\sup_{\substack{A'\in \mathscr{A}(\Omega) \\ A' \subset\subset A }} m_\rho(A',\mu) \ge  \int_{A\cap S_\rho\setminus \Sigma_0} f(\nabla u)
	+\int_{A\cap  \Sigma_0}[u]^p\,d\mu\,.
\]
Since $u\in \mathcal{V}_\rho(A)$, by the compactness of the trace operator, then it is admissible in the minimisation problem \eqref{mrhoa} and \eqref{inner} holds.
Now for every $p$-quasi-open set $U\subset \Omega_0$ and 
every $\rho\in\left(0,\tfrac{r_0}{2}\right)$, we also set 
\begin{equation}\label{hatm}
\begin{split}
\widehat m_\rho(U;\mu):=\inf_{{\substack{A\in \mathscr{A}(\Omega_0) \\ A\supset U }} } m_\rho(A;\mu)\,.
\end{split}
\end{equation}

In the following proposition we construct the measure on $p$-quasi open sets
by squeezing $\rho\to0^+$ in \eqref{hatm}. We point out that
the set function $\widehat m_\rho(U,\mu)$
is not sensitive to the change of the measure inside its equivalence class, by its very definition \eqref{hatm}. Hence,
according to Lemma~\ref{lm:criterionequiv}, the conclusion of Proposition~\ref{prop:cella2} cannot hold in general for a larger family than the collection of
all $p$-quasi open sets.

\begin{prop}\label{prop:cella2}
Let $\mu\in\mathcal{M}_p(\Omega_0;M)$. Then for every $p$-quasi open set $U\subset \Omega_0$ we have 
\begin{equation}
\label{cella2}
\mu(U) = \sup_{\rho>0}\widehat m_\rho(U;\mu)\,.
\end{equation}
\end{prop}
\begin{proof}
Let $U$ be a $p$-quasi open set. The inequality $\ge$ holds in \eqref{cella2} if for all fixed $\rho>0$
\begin{equation}\label{intrho}
\widehat m_\rho(U;\mu)\le I_\rho(U;\mu):= \inf_{v,A}\int_{U\cap S_{\rho}\setminus \Sigma_0}  f(\nabla v)\,dx+ \int_{U\cap \Sigma_0} [v]^p\,d\mu\,,
\end{equation}
where the infimum is taken among all pairs
$(v,A)$ where $A$ is an open set containing $U$ and $v\in W^{1,p}(A\cap S_\rho\meno \Sigma_0)$
is a non-negative function
with $v= 1$ $\mathscr{H}^{n-1}$-a.e.\ on $A\cap \Sigma_\rho^+$ and $v=0$ $\mathscr{H}^{n-1}$-a.e.\ on $A\cap \Sigma_\rho^-$. Indeed, given $\rho>0$ and an open set $A$ containing $U$, if we plug in $A$ and $v=1_{A\cap S_\rho}$ in the definition
of $I_\rho(U,\mu)$ from \eqref{intrho} we obtain that
$\widehat m_\rho(U,\mu)\le \mu(U)$.

Then we prove \eqref{intrho}. To do so, we fix $\eta>0$. Then there exists an admissible competitor $(v,A)$
for the minimisation problem \eqref{intrho} with
\begin{equation}\label{stima}
\int_{U\cap S_{\rho}\setminus \Sigma_0}  f(\nabla v)\,dx+ \int_{U\cap \Sigma_0} [v]^p\,d\mu \,<I_\rho(U;\mu )+\eta\,.
\end{equation}
Let $\delta>0$. Since $U$ is $p$-quasi open, there exists an open set $V_\delta$, with $C_{1,p}(V_\delta)<\delta$, such that $A_\delta=U\cup V_\delta$ is an open set. It is not restrictive to assume that $A_\delta\subset A$. Then there exists $ z_\delta\in W^{1,p}_0(\Omega)$ with
$z_\delta=1$ on $V_\delta\cap S_\frac{\rho}{2}$ and $\int_{\Omega} |\nabla z_\delta|^p\le \delta$.
Let also $\varphi\in C^\infty_0(\Omega) $ be a cut-off function with $\varphi=1$ in $S_{\frac{\rho}{2}}$ and $\varphi=0$ in $\Sigma_\rho^+$
and we define $w_\delta=1-\varphi z_\delta$.
Then 
$w_\delta =0$ on $S_{\frac{\rho}{2}}\cap V_\delta$, $w_\delta=1$ on $\Sigma_\rho^+$ and in addition $\int_{\Omega}|\nabla w_\delta|^p\le \omega(\delta)$, where $\omega(\delta)\to0^+$ as $\delta\to0^+$.
We choose $\delta$ small enough so that
\begin{equation*}
\int_{\Omega}|\nabla w_\delta|^p\le \eta\quad\text{and}\quad
\int_{A_\delta\cap S_{\rho}\setminus \Sigma_0}  f(\nabla v)\,dx \le \int_{U\cap S_{\rho}\setminus \Sigma_0}f(\nabla v)\,dx+\eta\,,
\end{equation*}
where the last inequality follows from the fact that $|V_\delta|\to0^+$ as $\delta\to0^+$.

We observe that $v_\delta=v\wedge w_\delta\in W^{1,p}(A_\delta\cap S_\rho\setminus \Sigma_0)$ is a non-negative function with
$v_\delta=1$ $\mathscr H^{n-1}$ a.e.\ on $\Sigma_\rho^+\cap A_\delta$ and $v_\delta=0 $
$\mathscr H^{n-1}$ a.e.\ on $\Sigma_\rho^-\cap A_\delta$. Then, by \eqref{hatm} we have
\begin{equation}
\label{eqdeltaU}
\widehat m_\rho(U;\mu)\le 
\int_{A_\delta\cap S_{\rho}\setminus \Sigma_0}  f(\nabla v_\delta)\,dx+ \int_{A_\delta\cap \Sigma_0} [v_\delta]^p\,d\mu \,.
\end{equation}
As for the volume integral we observe that
\begin{equation}\label{volinte}
\begin{split}
\int_{A_\delta\cap S_{\rho}\setminus \Sigma_0}  f(\nabla v_\delta)\,dx
& \le
\int_{A_\delta\cap S_{\rho}\setminus \Sigma_0}  f(\nabla v)\,dx+
\int_{A_\delta\cap S_{\rho}\setminus \Sigma_0}  f(\nabla w_\delta)\,dx\,,\\
& \le\int_{U\cap S_{\rho}\setminus \Sigma_0}f(\nabla v)\,dx + (1+\Lambda)\eta\,.
\end{split}
\end{equation}
As for the integral on $\Sigma_0$ we notice that $[v_\delta]\le [v]$ $p$-q.e.\ on $U\cap \Sigma_0$ and
$[v_\delta]=0$ $p$-q.e. on $V_\delta\cap \Sigma_0$. This implies that
\begin{equation}
\label{surfinte}
\int_{A_\delta\cap \Sigma_0} [v_\delta]^p\,d\mu \le
\int_{U\cap \Sigma_0}[v]^p\,d\mu \,.
\end{equation}
By \eqref{stima}, \eqref{eqdeltaU}, \eqref{volinte}, and \eqref{surfinte} we deduce that
\[
\widehat m_\rho(U;\mu)\le I_\rho(U;\mu )+(2+\Lambda) \eta\,.
\]
Since $\eta$ is arbitrary we get \eqref{intrho}.

To prove the reverse inequality $\le$ in \eqref{cella2}, we fix an arbitrary sequence $\rho_i\to0^+$. By \eqref{hatm}, there exists a sequence of open sets $(A_i)$ containing $U$ such that
\begin{equation}
\label{qa1}
m_{\rho_i}(A_i,\mu)<\widehat m_{\rho_i}(U,\mu) +\rho_i\,.
\end{equation}
Let also $u_i\in \mathcal{V}_{\rho_i}(A_i)$ be such that
\begin{equation}
\label{qa2}
	m_{\rho_i}(A_i,\mu) = \int_{A_i\cap S_{\rho_i}\setminus \Sigma_0} f(\nabla u_i)\,dx
	+\int_{A_i\cap \Sigma_0} [u_i]^p\,d\mu\,.
\end{equation}

Let $\psi:\Sigma_0\to[0,\infty]$ be a Borel function and $\sigma$ be a Radon measure
of class $W^{-1,p'}(\Omega_0)$ such that the measure $\psi\sigma$ is equivalent to $\mu$. We
recall that we can find such $\psi$ and $\sigma$ thanks to Lemma~\ref{lm:equivmis}. For all $k>0$
the truncated measure $\psi^k\sigma$ is a Radon measure of class  $W^{-1,p'}(\Omega_0)$.
Thus, given a sequence $z_i$ converging to a function $z$ weakly in $W^{1,p}_0(\Omega_0)$,
for all fixed $k>0$ we have that
\begin{equation}
\label{qa3}
\lim_{i\to\infty}\int_{\Sigma_0} \tilde z_i \psi^k \,d\sigma=
\lim_{i\to\infty}\langle \psi^k\sigma, z_i\rangle = 
\langle \psi^k\sigma,z\rangle= \int_{\Sigma_0} \tilde z\psi^k\,d\sigma\,.
\end{equation}
For every $i$ we set
 \[
v_{i} (x)=\begin{cases}
u_{i}(x)\,, & \qquad \text{if $x\in A_i\cap S_{\rho_i}^+$,}\\
u_{i}( \mathbf{R}x)
& \qquad \text{if $x\in A_i\cap S_{\rho_i}^-$\,,}\\
1\,,& \qquad \text{if $x\in A_i\setminus S_{\rho_i}$,}
\end{cases}\quad
w_{i} (x)=\begin{cases}
u_{i}( \mathbf{R}x) & \qquad \text{if $x\in A_i\cap S_{\rho_i}^+$\,,}\\
u_{i}(x)\,, & \qquad \text{if $x\in A_i\cap S_{\rho_i}^-$,}\\
0\,,& \qquad \text{if $x\in A_i\setminus S_{\rho_i}$,}
\end{cases}
\]
where
\[
\mathbf{R}x:= x-\nu_0((x-x_0)\cdot\nu_0) + 2\phi(x-\nu_0((x-x_0)\cdot \nu_0))-(x-x_0)\cdot\nu_0 
\]
is the reflection about the disk $\Pi_0=\{x\in \Omega_0\colon (x-x_0)\cdot \nu_0=0\}$,
$x_0$ and $\nu_0$ are the center and the axis of the cylinder $\Omega_0$, and $\phi$
is the function describing $\Sigma_0$ as a graph according to \eqref{Ngraph}.

We take a sequence of functions $\chi_k$ as in Lemma~\ref{lm:qa0}.
Then $z_{i,k}= |v_i-w_i|^p\chi_k$ defines an equibounded sequence in $W^{1,p}_0(\Omega_0)$.
Indeed, by Lemma~\ref{lm:qa0} the function $z_{i,k}$ vanishes out of a compact set contained in $U$, hence
in $\Omega_0$;
moreover, both $v_i$ and $w_i$ are defined from $u_i$ by a norm-preserving endomorphism of
$W^{1,p}(\Omega_0)$ and $u_i$ is equibounded in $W^{1,p}(\Omega_0)$ by \eqref{Ac}, \eqref{qa1},
and \eqref{qa2}. Therefore, by possibly passing to a subsequence (not relabelled) we may assume that
$z_{i,k}\to \chi_k $, as $i\to\infty$, weakly in $W^{1,p}_0(\Omega_0)$. Since
the $p$-quasi continuous representative of $z_{i,k}$ is $[u_i]^p\tilde \chi_k$ from \eqref{qa3} we deduce
that
\begin{equation}
\label{qa4}
\int_{U\cap \Sigma_0}\tilde \chi_k\psi^k\,d\sigma
=
\lim_{i\to\infty} \int_{A_i\cap \Sigma_0} [u_i]^p\tilde \chi_k\psi^k\,d\sigma
\le \lim_{i\to\infty} \int_{A_i\cap \Sigma_0} [u_i]^p\,d\mu\,,
\end{equation}
where we also used the fact that according to Lemma~\ref{lm:qa0} $\chi_k=0$ $p$-q.e.\ in $\Omega_0\setminus U$
and the equivalence of $\mu$ to $\psi\sigma$.

Combining \eqref{qa4} with \eqref{qa1} and \eqref{qa2}, for every $k>0$, we obtain
\begin{equation*}
\int_{U\cap \Sigma_0}\tilde \chi_k
\psi^k\,d\sigma
\le \lim_{i\to\infty} \widehat m_{\rho_i}(U,\mu)\,.
\end{equation*}
Since by Lemma~\ref{lm:qa0} the sequence $\chi_k$ converges to
$1$ $p$-q.e.\ in $U$, letting $k\to\infty$ in the above inequality, by the monotone convergence theorem we deduce that
\begin{equation*}
	\mu(U)\le \lim_{i\to\infty} \widehat m_{\rho_i}(U,\mu)\,,
\end{equation*}
up to subsequences. The desired inequality follows from the arbitrariness of the infinitesimal sequence $(\rho_i)$ and the proposition is proved.
\end{proof}

\subsection{Equivalence with localised minimum problems}

Let $(K_j)$ be a sequence of compact sets satisfying \eqref{eq:1.1sec3}. 
We take  $\rho_j=\max\{{\rm dist}(x,M)\colon x\in K_j\}$, so that $\rho_j\to0^+$ as $j\to\infty$. 
For $A\in\mathscr{A}(\Omega_0)$, $\rho\in\left(0,\tfrac{r_0}{2}\right)$, and for all $j$ with
$\rho_j<\rho$ 
we set
\begin{equation}
\label{4}
m_{\rho,j}(A) :=
\displaystyle\min_{v\in\mathcal{V}_{\rho,j}(A)}
\int_{A\cap S_\rho\meno K_j} f(\nabla v) \,dx \,,
\end{equation}
where
$\mathcal{V}_{\rho,j}(A)$ is the set of all functions $v\in L^{1,p}(A\cap S_\rho\meno K_j)$
such that $v= 1$ $\mathscr{H}^{n-1}$-a.e.\ on $\Sigma_\rho^+\cap A$
and $v=0$ $\mathscr{H}^{n-1}$-a.e. on $\Sigma_\rho^-\cap A$. The existence of
the minimum in \eqref{4} follows by the direct methods of the Calculus of Variations. By truncation every minimiser of \eqref{4} is bounded by 0 and 1 and then it belongs to $W^{1,p}(A\cap S_\rho\meno K_j)$. Moreover, 
we set
\begin{equation}\label{m''}
m'_\rho(A) :=\liminf_{j\to\infty} m_{\rho,j}(A)\quad
\text{and}\quad
m''_\rho(A) :=\Limsup_{j\to\infty} m_{\rho,j}(A)\,.
\end{equation}

We have the following result.
\begin{lm}\label{ultimolemma}
Let $A', A''\in \mathscr A(\Omega)$ with $A'\subset \subset A''$.
Assume that $(\mathscr F_{K_j}(\cdot,A'))$ $\Gamma$-converge in $L^0(\Omega)$ to $\mathscr F^\mu(\cdot,A')$ and $(\mathscr F_{K_j}(\cdot,A''))$ $\Gamma$-converge in $L^0(\Omega)$ to $\mathscr F^\mu(\cdot,A'')$.
Then, for every $\rho>0$
\[
 m_\rho'(A')\ge m_\rho(A',\mu) \quad \text{and} \quad m''_\rho(A')\le m_\rho(A'',\mu)\,.
\]
\end{lm}
\begin{proof}
To prove the first inequality, let $(u_j)\subset L^{1,p}(A' \meno K_j)$ be a sequence of minimisers of the minimum problems \eqref{4}. 
By \eqref{Ac} it is not restrictive to assume this sequence to be equibounded in $W^{1,p}(A' \meno K_j)\cap L^\infty(A')$. Then there exists a function $u\in W^{1,p}(A'\meno M)\cap L^\infty(A')$ and a subsequence of $(u_j)$, not relabelled, 
converging to $u$ weakly in $W^{1,p}(A'\meno M_r)$ for every $r>0$,
where 
\(
M_r = \left\{ x\in \Omega \colon {\rm dist}(x,\Sigma)\le\rho \right\}
\). Therefore, $(u_j)$ converges  in $L^0(\Omega)$  to a function $u\in L^\infty(A')$.
Moreover, by the $\Gamma$-convergence assumption
\begin{equation*}
m'_\rho(A')=\liminf_{j\to\infty} \int_{A'\cap S_\rho \setminus K_j} f(\nabla u_j)\,dx\ge \int_{A'\cap S_\rho\setminus M}f(\nabla u)\,dx+ \int_{A' \cap M} [u]^p\,d\mu\,.
\end{equation*}
Since the trace operator is compact, the function $u$ belongs to $\mathcal V_\rho(A')$ and thus it is admissible in the minimisation problem \eqref{mrhoa}. This proves the first inequality.

To prove the the second inequality we assume $m_\rho(A'',\mu)$ to be finite and let $u$ one of its minimisers. By the $\Gamma$-convergence assumption there exists a sequence $(u_j)\subset W^{1,p}(A''\setminus K_j)\cap L^\infty(A'')$ converging to $u$ in $L^0(\Omega)$ such that
\begin{equation}
\label{recseqqui}
m_\rho(A'',\mu)=\mathscr F^\mu(u,A'') = \lim_{j\to\infty}  \int_{A''\cap S_\rho \meno K_{j}} f(\nabla u_j)\,dx \\\,. 
\end{equation}
We consider now a cut-off function $\varphi\in C^{\infty}_c(S_\rho)$ such that $0\le\varphi\le1$, $\varphi=1$
in $S_{\rho/2}$, $\varphi=0$ out of $S_\rho$, and $|\nabla\varphi |\le \frac{C}{\rho}$ for a suitable constant $C>0$. We consider the function
$v_j =  \varphi u_j+ (1-\varphi)u$ so that $v_j\in \mathcal{V}_{\rho,j}(A')$ and $\nabla v_j = (\varphi \nabla u_j+(1-\varphi)\nabla u)+(u_j-u)\nabla \varphi$.
By \eqref{4} and \eqref{A} for every $\varepsilon>0$
we obtain
\begin{equation}\label{13.1}
\begin{split}
& m_{\rho,j}(A')\le \int_{A'\cap S_\rho \meno K_{j}} f(\nabla v_j)\,dx \\
&\le \frac{1}{{(1-\varepsilon)}^{p-1}}\int_{A'\cap S_\rho \setminus K_j}f(\varphi \nabla u_j +(1-\varphi)\nabla u)\, dx+\frac{C^p\Lambda}{\rho^p \varepsilon^{p-1}}\int_{A'\cap S_\rho \setminus K_j}|u_j-u|^p\, dx\,.\\
\end{split}
\end{equation}
Next, we osberve that by \eqref{Ab} we have
\[
\int_{A'\cap S_\rho \setminus K_j}f(\varphi \nabla u_j +(1-\varphi)\nabla u)\, dx\le
\int_{A''\cap S_\rho \setminus K_j}\varphi f(\nabla u_j )\, dx+
\int_{A''\cap S_\rho \setminus K_j}(1-\varphi)f(\nabla u)\, dx\,,
\]
where we also used that $A'\subset A''$. Let now $\widehat A\in \mathcal A(\Omega)$ be a Lipschitz set with $A'\subset \widehat A \subset A''$. Then, 
$u_j$ converges to $u$ in $L^p(\widehat A)$, hence
\[
\lim_{j\to\infty}\int_{\widehat{A}\cap S_\rho \setminus K_j}|u_j-u|^p\, dx=0\,.
\]
Using the last two inequalities and the fact that $\varphi\le1$, from \eqref{13.1} we obtain
\[
\limsup_{j\to\infty}m_{\rho,j}(A') \le\frac{1}{{(1-\varepsilon)}^{p-1}} \Big[\limsup_{j\to\infty} \int_{A''\cap S_\rho \setminus K_j} f(\nabla u_j )\, dx
+ \int_{A''\cap S_\rho \setminus M}(1-\varphi)f(\nabla u)\, dx\Big]\,.
\]
Since $\varphi$ was arbitrary, by the monotone convergence theorem we deduce that
\[
\limsup_{j\to\infty} m_{\rho,j}(A')\le \frac{1}{{(1-\varepsilon)}^{p-1}} \limsup_{j\to\infty} \int_{A''\cap S_\rho \meno K_{j}} f(\nabla u_j)\,dx\,.
\]
This combined with \eqref{recseqqui} and the arbitrariness of $\varepsilon$ provides the conclusion.
\end{proof}

\begin{teo}\label{equiteo}
The following three conditions are equivalent.
\begin{enumerate}[a)]
\item \emph{($\Gamma$-convergence in $\Omega$).} The sequence of functionals $(\mathscr F_{K_j}(\cdot, \Omega))$ defined by \eqref{eq:1.2} $\Gamma$-converge in $L^0(\Omega)$
 to a functional $\mathscr F:L^0(\Omega)\to[0,\infty]$.
\item \emph{($\Gamma$-convergence on a rich family).} There exist a measure $\mu\in \mathcal{M}_p(\Omega;M)$ and a rich family $\mathscr R\subset \mathscr A(\Omega)$ such that the sequence of functionals $(\mathscr F_{K_j}(\cdot,A))$ defined by \eqref{eq:1.2} $\Gamma$-converges in $L^0(\Omega)$ to the functional $\mathscr F^\mu(\cdot,A)$ defined by \eqref{FmuA} for all $A\in\mathscr R$. 
\item \emph{(Equality condition of localized minimum problems).} Let $\{\Omega_i\}$ be an open covering of $M$ consisting of cylinders of the form \eqref{cilynder} for some $x_i\in M$, $\nu_i\in \mathbb{R}^n$ with $|\nu_i|=1$ and, $r_i>0$. Then
\begin{equation}\label{thirdcondition}
 \sup_{\substack{A'\in \mathscr{A}(\Omega) \\ A' \subset\subset A }} m_\rho'(A')= \sup_{\substack{A'\in \mathscr{A}(\Omega) \\ A' \subset\subset A }} m_\rho''(A')\,,
\end{equation}
for every open set $A\subset \Omega_i$ and every $\rho\in (0,\frac{r_i}{2})$. Here $m'$ and $m''$ are defined by \eqref{m''} with $\Omega_i$ in place of $\Omega_0$ and the graph $\Sigma_i=\Omega_0\cap \Sigma$ defined as in \eqref{Ngraph}.
\end{enumerate}
\end{teo}

\begin{proof}
We assume (a) and prove (b).
To this aim, we fix a subsequence (not relabelled) of compact sets $(K_j)$
and we consider the sequence of functionals $(\mathscr F_{K_j})$ defined by \eqref{eq:1.2} accordingly.
By Theorem~\ref{mainthm} there exist a measure $\mu\in\mathcal{M}_p(\Omega;M)$ and a rich family $ \mathscr R\subset \mathscr A(\Omega)$ such that for every $A\in\mathscr R$ a subsequence of  $(\mathscr F_{K_j}(\cdot,A))$  $\Gamma$-converge in $L^0(\Omega)$ to the functional $\mathscr F^{\mu}(\cdot,A)$ defined by \eqref{FmuA}. Thanks to the Urysohn property (see~\cite[Proposition 8.3]{DMb}), to get
the conclusion it is enough to make sure that $\mu$ is independent on the subsequence chosen at the beginning.
This follows since by assumption the $\Gamma$-limit of $(\mathscr F_{K_j}(\cdot, \Omega))$ do not depend on the choice of the subsequence; hence the measure $\mu$ is uniquely determined (up to equivalence).

We now assume (b) and prove (c). By Lemma~\ref{ultimolemma} and \eqref{m''} for every $A\in \mathscr A(\Omega)$ and every $A',A''\in\mathscr R$ such that $ A'\subset\subset A''\subset\subset A\subset\subset \Omega_i$ we have 
\[
m_\rho(A',\mu)\le m_\rho'(A')\le m_\rho''(A')\le m_\rho(A'',\mu)\le m_\rho(A,\mu).
\]
These inequalities combined with \eqref{inner} and \cite[Remark 14.9 and 14.13]{DMb} yield
\[
m_\rho(A,\mu)= \sup_{\substack{A'\in \mathscr{R} \\ A' \subset\subset A }} m_\rho(A',\mu)\le  \sup_{\substack{A'\in \mathscr{R} \\ A' \subset\subset A }} m_\rho'(A')\le \sup_{\substack{A'\in \mathscr{R} \\ A' \subset\subset A }} m_\rho''(A')\le m_\rho(A,\mu)\,,
\]
so that these inequalities are in fact equalities and \eqref{thirdcondition} holds.

We conclude the chain of implications by assuming c) and proving a). To this aim, we fix a subsequence (not relabelled) of compact sets $(K_j)$
and we consider the sequence of functionals $(\mathscr F_{K_j}(\cdot, \Omega))$ defined by \eqref{eq:1.1sec3} accordingly. 
By Theorem~\ref{mainthm} there exist a measure $\mu\in\mathcal{M}_p(\Omega;M)$ and a rich family $ \mathscr R\subset \mathscr A(\Omega)$ with $\Omega\in\mathscr R$ such that
 a subsequence of
 $(\mathscr F_{K_j}(\cdot,A))$
 $\Gamma$-converge
in $L^0(\Omega)$ to the functional $\mathscr F^{\mu}(\cdot,A)$ for every $A\in\mathscr R$. 
To get the thesis stated in a), by the Urysohn property (see~\cite[Proposition 8.3]{DMb}) it is again sufficient to prove that the measure $\mu$ is independent on the subsequence chosen at the beginning. To see this first notice that, 
by arguing as in the previous implication one deduces that $m_\rho(A,\mu)$ as defined in \eqref{mrhoa} coincides with \eqref{thirdcondition} for all open sets $A\subset \Omega_i$ and all $\rho\in (0,\frac{r_i}{2})$. Moreover, passing to subsequences, the value $m'_\rho(A)$ can only increase while $m''_\rho(A)$ only decrease (see \eqref{m''}). This with the fact that $m'_\rho(A)\le m''_\rho(A)$ implies that the assumption \eqref{thirdcondition} is not affected by the passage to subsequences and so also the measure $\mu$.  The theorem is then proved.
\end{proof}

As a consequence of the results proved in this section we deduce a cell formula to construct every limit measure  appearing in the conclusion of Theorem~\ref{mainthm}.

\begin{cor}
Let $\mu\in \mathcal M_p(\Omega;M)$. Assume that $(\mathscr F_{K_j}(\cdot,\Omega))$ $\Gamma$-converge in $L^0(\Omega)$ to $\mathscr F^\mu(\cdot,\Omega)$. 
Then there exists a rich family of open sets $\mathscr R\subset \mathscr A(\Omega)$ such that, for every $p$-quasi open set $U\subset \Omega_0$ with $\Omega_0$ a cylinder as in \eqref{cilynder} with $x\in M$ it holds  
\[
\mu(U)=\sup_{\rho>0} \inf_{\substack{A\in \mathscr A(\Omega_0)\\ A\supset U}} \sup_{\substack{A'\in \mathscr R\\ A'\subset \subset A}} \lim_{j\to\infty} \min_{u\in \mathcal V_{\rho, j}(A')} \int_{A'\cap S_\rho\setminus K_j} f(\nabla u)\, dx\,,
\]
where $\mathcal{V}_{\rho,j}(A)$ denotes the set of all non-negative functions $v\in L^{1,p}(A\cap S_\rho\meno K_j)$
with $v= 1$ $\mathscr{H}^{n-1}$-a.e.\ on $A\cap \Sigma_\rho^+$
and $v=0$ $\mathscr{H}^{n-1}$-a.e. on $A\cap \Sigma_\rho^-$.
\end{cor}
\begin{proof}
The result follows by Proposition~\ref{prop:cella2} and by Theorem~\ref{equiteo}.
\end{proof}


\begin{thebibliography}{99}
\bibitem{AFP} L. Ambrosio, N. Fusco, D. Pallara. {Functions of bounded variation and free discontinuity problems}.
The Clarendon Press, Oxford University Press, New York, 2000.
\bibitem{ansini} N. Ansini. The Nonlinear Sieve Problem and Applications to Thin Films. Asymptotic Analysis 39 (2) (2004), 113-145.
\bibitem{AB} N. Ansini and A. Braides. Separation of scales and almost-periodic effects in the asymptotic behaviour of perforated periodic media. Acta Appl. Math. 65 (2001), 59-81.
\bibitem{AB2} N. Ansini and A. Braides. Asymptotic analysis of periodically-perforated nonlinear media, J. Math. Pures Appl. 81 (2002), 439-451.
\bibitem{AP} H. Attouch and C. Picard. Comportement limite de probl\`emes de trasmission unilateraux \`a travers des grilles de forme
quelconque. Rend. Sem. Mat. Univ. Politec. Torino 45 (1987), 71-85.
\bibitem{C1} C. Conca. On the application of the homogenization theory to a class of problems arising in fluid mechanics. J. Math.
Pures Appl. 64 (1985), 31-75.
\bibitem{C2} C. Conca. \'Etude d'un fluide traversant une paroi perfor\'ee I. Comportement limite pr\`es de la paroi. J. Math. Pures Appl.
66 (1987), 1-43.
\bibitem{C3} C. Conca. \'Etude d'un fluide traversant une paroi perfor\'ee II. Comportement limite loin de la paroi. J. Math. Pures Appl.
66 (1987), 45-69.
\bibitem{Cor} G. Cortesani. Asymptotic behaviour of a sequence of Neumann problems. {Comm. in Part. Diff. Eq.} {22} (9-10) (1997), 1691-1729.
\bibitem{DMb} G. Dal Maso. {An introduction to $\Gamma$-convergence.} Progress in Nonlinear Differential Equations and their Applications, 8. Birkh\"auser Boston, Inc., Boston, MA, 1993.
\bibitem{DM}  G. Dal Maso. On the integral representation of certain local functionals. { Ricerche Mat.} { 32} (1) (1983), 85-113.
\bibitem{DM2} G. Dal Maso. $\Gamma$-convergence and $\mu$--capacities. { Ann. Scuola Norm. Sup. Pisa Cl. Sci.} (4) { 14} (3) (1987), 423--464.
\bibitem{DMD} G. Dal Maso, A. Defranceschi. Limits of nonlinear problems in varying domains. { Manuscripta Math.} { 61} (1988), 251-278.
\bibitem{DMI}
G. Dal Maso, F. Iurlano. Fracture models as $\Gamma$-limits of damage models. { Commun. Pure Appl. Anal.} { 12} (2013), 1657-1686.
\bibitem{D} A. Damlamian, Le probl\`eme de la passoire de Neumann. Rend. Sem. Mat. Univ. Politec. Torino 43 (1985), 427-450.
\bibitem{DV} T. Del Vecchio. The thick Neumann's sieve. Ann. Mat. Pura Appl. 147 (1987), 363-402.
\bibitem{DL} {J. Deny, J.L. Lions}. {Les espaces du type Beppo Levi.} {Ann. Inst. Fourier (Grenoble)} {5} (1953), 305-370.
\bibitem{EG} L.C. Evans, R. Gariepy. { Measure theory and fine properties of functions.} Studies in Advanced Mathematics. CRC Press, Boca Raton, FL, 1992.
\bibitem{HKM} J. Heinonen, T. Kilpel\"ainen, O. Martio. { Nonlinear Potential Theory of Degenerate Elliptic Equations.} Clarendon Press, Oxford 1993.
\bibitem{Mazya} V.G. Maz'ya. { Sobolev Spaces. With Applications to Elliptic Partial Differential Equations.} Grundlehren der Mathematischen Wissenschaften 342 (2nd ed.), Berlin-Heidelberg-New York 2011.
\bibitem{M} {F. Murat}. {The Neumann sieve.} {Nonlinear Variational Problems (Isola d'Elba, 1983)},
Res. Notes in Math. 127, Pitman, London, (1985), 24-32.
\bibitem{P} C. Picard. Analyse limite d'\'equations variationnelles dans un domaine contenant une grille. RAIRO Mod\'el. Math. Anal.
Num\'er. 21 (1987), 293-326.
\bibitem{SP1} E. Sanchez-Palencia. Non-Homogeneous Media and Vibration Theory. Lecture Notes in Physics 127, Springer-
Verlag, Berlin, 1980.
\bibitem{SP2} E. Sanchez-Palencia. Boundary value problems in domains containing perforated walls. In: Nonlinear Partial Differential
Equations and Their Applications. Coll\'ege de France Seminar III, Res. Notes in Math. 70, Pitman, London,
(1981), 309-325.
\bibitem{SP3} E. Sanchez-Palencia. Un probl\`eme d'\'ecoulement lent d'un fluide visqueux incompressible au travers d'une paroi finement
perfor\'ee. In: Les M\'ethodes de l'Homog\'en\'eisation: Th\'eorie et Applications en Physique, Collection
de la Direction des \'Etudes et Recherches d'\'Electricit\'e de France 57 (1985), 371-400.
\bibitem{Z} W.P. Ziemer. { Weakly Differentiable Functions}. Springer-Verlag, Berlin 1989.
\end{thebibliography}
\end{document}